\documentclass[12pt]{amsart}
\pdfoutput=1
\usepackage{amssymb,amsthm,amsmath,amscd}
\usepackage{mathtools}
\usepackage{soul}
\usepackage{kbordermatrix}

\RequirePackage[usenames]{color}

\input{kmacros3.sty}

\usepackage{tikz-cd}

\usepackage{graphicx}
\usepackage[all,cmtip]{xy}

\usepackage{mabliautoref}
\numberwithin{equation}{theorem}

\usepackage{bm}
\usepackage{pifont}
\usepackage{upgreek}

\usepackage{eucal}
\usepackage{ulem}
\usepackage{stmaryrd}

\numberwithin{equation}{theorem}

\usepackage{fullpage}

\usepackage{setspace}

\usepackage[nameinlink]{cleveref}
\usepackage[hyperpageref]{backref}
\renewcommand*{\backrefalt}[4]{{%
    \ifcase #1 Not cited.%
          \or (cit. on p.~#2).%
          \else (cit. on pp.~#2).%
    \fi%
    }}
\usepackage{enumitem}
\usepackage{calc}

\usepackage{bbding}

\usepackage{verbatim}
\usepackage{alltt}

\DeclareMathOperator{\lct}{lct}
\DeclareMathOperator{\divi}{div}

\DeclareMathOperator{\LT}{in}
\newcommand{\spair}{\mathcal S}
\DeclareMathOperator{\nvol}{\widehat{vol}}
\newcommand{\wtl}[1]{\widetilde{#1}}

\begin{document}

\title{Limit $F$-signature functions of \\ two-variable binomial hypersurfaces}

\author{Anna Brosowsky}
\address{Department of Mathematics, University of Nebraska--Lincoln, Lincoln, NE 68588, USA}
\email{abrosowsky2@unl.edu}
\thanks{Brosowsky was supported in part by NSF grants DMS \#1840234 and \#2101075 and by NSF postdoctoral fellowship \#2402293.}

\author{Izzet Coskun}
\address{Department of Mathematics, University of Illinois at Chicago, Chicago, IL 60607, USA}
\email{icoskun@uic.edu}
\thanks{Coskun was supported in part the NSF grant DMS-2200684.}

\author{Suchitra Pande}
\address{Department of Mathematics, University of Utah, Salt Lake City, UT 84112, USA}
\email{suchitra.pande@utah.edu}
\thanks{Pande was supported in part by NSF grants \#1952399, \#1801697 and \#2101075.}

\author{Kevin Tucker}
\address{Department of Mathematics, University of Illinois at Chicago, Chicago, IL 60607, USA}
\email{kftucker@uic.edu}
\thanks{Tucker was supported in part by NSF Grant DMS \#2200716.}


\begin{abstract}
The $F$-signature is a fundamental numerical invariant of singularities in positive characteristic. Its positivity detects strong $F$-regularity, an important class of singularities related to KLT singularities in characteristic zero. 
In this paper, we compute the limiting $F$-signature function of binomial and other related hypersurfaces in two variables as the characteristic $p \to \infty$. In particular, we show it is a piecewise polynomial function, and relate it to the normalized volume.
\end{abstract}

\maketitle

\section{Introduction}
The $F$-signature is a fundamental numerical invariant of singularities in positive characteristic. Its positivity detects strong $F$-regularity, an important class of singularities related to KLT singularities in characteristic zero birational algebraic geometry.
In this paper, we compute the limiting $F$-signature function of binomial and other related hypersurfaces in two variables as the characteristic $p \to \infty$. 

Fix a prime $p > 0$ and let $A_{p}$ denote either $ \mathbb{F}_{p}\llbracket x_{1}, \ldots, x_{n} \rrbracket$  the power series ring in $n$ variables over the finite field $\mathbb{F}_{p}$, or the local ring $\mathbb{F}_p [x_1, \dots , x_n]_{(x_1, \dots, x_n)}$. Let $R = A_p / I$ for an ideal $I \subseteq A_p$. It is natural to study the singularities of $R$ using the behavior of the Frobenius or $p$-th power endomorphism. By the work of Kunz \cite{Kunz.69}, $R$ is regular if and only if Frobenius is flat, and in general there are a number of numerical invariants of singularities which aim to measure the failure of flatness for the iterated Frobenius. Central among these are the Hilbert-Kunz multiplicity $e_{\mathrm{HK}}(R)$ \cite{Monsky.83} and $F$-signature $s(R)$ \cite{Smith+VandenBergh.97,Huneke+Leuschke.02, TuckerFsigExists}, which can be expressed as 
\begin{equation*}
    e_{\mathrm{HK}}(R) = \lim_{e \to \infty} \frac{\min \left\{ m \mid \exists R^{\oplus m} \twoheadrightarrow F^e_*R \right\}}{p^{e\dim(R)}}  \quad \mbox{and} \quad s(R) = \lim_{e \to \infty} \frac{\max \left\{ m \mid \exists F^e_*R \twoheadrightarrow R^{\oplus m} \right\}}{p^{e\dim(R)}}.
\end{equation*}
Provided $R$ is equidimensional, both $e_{\mathrm{HK}}(R)$ and $s(R)$  are $1$  if and only if $R$ is regular, and $s(R) > 0$ if and only if $R$ is strongly $F$-regular \cite{WatanabeYoshidaHKMultandInequality, Huneke+Leuschke.02, AberbachLeuschke}. On the other hand, these invariants are
typically very hard to compute and the precise value is known only for certain limited classes
of rings. 

In birational algebraic geometry, it is crucial to consider the singularities of divisor pairs. The notion of $F$-signature has been extended to the setting of divisor pairs in \cite{BlickleSchwedeTuckerFSigPairs1,BlickleSchwedeTuckerFsigofPairs2}, where the positivity of the $F$-signature again characterizes strong $F$-regularity. In the case of a hypersurface pair $(A_p, f^{t})$, where $t \in \mathbb{R}_{\geq 0}$ and $f \in A_{p}$ is contained in $\mathfrak{m} = (x_1, \dots, x_n)$, we have
\begin{align*}
    s(A_p, f^{t}) &= \lim_{e \to \infty} \frac{1}{p^{en}} \ell \left( \frac{A_p}{(x^{p^e}_{1}, \ldots, x^{p^e}_{n}):f^{\lceil tp^e \rceil}} \right) \\ &= 1 - \lim_{e \to \infty} \frac{1}{p^{en}} \ell \left( \frac{A_p}{(x^{p^e}_{1}, \ldots, x^{p^e}_{n},f^{\lceil tp^e \rceil})} \right).
\end{align*}
 While this non-increasing function is continuous \cite[Theorem 3.3]{BlickleSchwedeTuckerFsigofPairs2} and even convex \cite[Theorem 3.9]{BlickleSchwedeTuckerFsigofPairs2},  
 one cannot expect the derivative of this function to behave well; see \cite[Remark 4.11]{BlickleSchwedeTuckerFsigofPairs2} for such an example. More generally, when
 $n =\dim(R)= 2$, the function $\frac{a}{b} \mapsto s(R,f^{\frac{a}{p^b}})$ is a $p$-fractal in the sense of \cite{Monsky+Teixeira.04,Monsky+Teixeira.06} (see \cite[Section 4]{BlickleSchwedeTuckerFsigofPairs2}).

 The $F$-signature function also encodes a number of important invariants of the hypersurface $R = A_p / (f)$, namely we have that 
 \[
 \partial_+ s(A_p, f^{0}) = - e_{HK}(R), \qquad \partial_- s(A_p, f^{1}) = - s(R),
 \]
 and moreover the least $t$-intercept of $s(A_p, f^{t})$ is the $F$-pure threshold of the pair $(A_p,f)$.

 While the $F$-signature in a fixed characteristic $p$ is  not fully understood, it is hoped that limiting as the characteristic $p \to \infty$ might produce a more tractable invariant with additional geometric meaning. However, other than for diagonal hypersurfaces, few cases have been computed (see \cite{caminatashidelertuckerzermandiagonalhypersurfaces}). In this article, we add another class of examples in two variables. Our examples include binomials, and geometrically they encompass both three points in $\mathbb P^1$ and the complete ADE (or simple) plane curve singularities (see \cite[Section 4.3]{LeuschkeWiegandCMReps}). 

 \begin{theoremA*}
 \mbox{}
 [\mbox{\Cref{thm:limitfsig}, \textit{cf.} \cite[Section 2]{HanThesis}}]
    Let $g = x^a y^b (x^u + y^v) ^c $ for  non-negative integers $a,b,c,u$, and $v$. Assume that $u$, $v$, and $c$ are positive. We consider $g$ as a polynomial in $\ZZ[x,y]$, so that we may consider the reduction of $g$ to characteristic $p >0$ for each prime $p$.
    Set $\lambda_0 = \frac{\frac{1}{u} + \frac{1}{v}}{\frac{a}{u} + \frac{b}{v} + c}$, and $\lambda = \min \{\frac{1}{a}, \frac{1}{b}, \frac{1}{c}, \lambda_0 \}$.
    Let $\psi(t)$ be the limit $F$-signature function of $g$, which is defined by
    \[ \psi(t):= \lim_{p \to \infty} s(\mathbb{F}_p\llbracket x,y \rrbracket, g^t).  \]
    Then, for $t < \lambda$, $\psi(t)$ is given by the following formula:
    \begin{enumerate}
        \item If $tav +u \geq tbu+v+cuvt$, then $\psi(t) = ( 1 -at) (1- (b +cv)t)$.
        \item If $tbu +v \geq tav+u+cuvt$, then $\psi(t) = (1-bt) (1 - (a +cu)t)$.
        \item If $tcuv + u + v \geq  tav  + tbu +2uv $, then $\psi(t) = (1-ct) (v(1-at) + u (1-bt) -uv)$.
        \item Otherwise, $\psi(t) = \frac{1}{4uv} (u+v - t(av + bu+ cuv))^2$.
    \end{enumerate}
\end{theoremA*}

\noindent
In fact, our methods are more precise and give an explicit computation of $s(\mathbb{F}_p\llbracket x,y \rrbracket,g^{t})$ for any $p$ and $t \in \frac{1}{p} \mathbb{Z}_{\geq 0}$ when $u=v=1$, see \Cref{fsigfunction-pdenom}.
This is achieved by giving an explicit Gr\"obner basis for ideals  of the form $I= ( x^M, y^N, f^K)$ with $f=x^ay^b(x+y)^c$, for $M,N,K,a,b,c$  all non-negative integers, and $M+(c-a)K \leq p$; see \Cref{prop:Grobnerbasis}. 

On the other hand, if one is solely interested in the length computations necessary for Theorem~A, one can do a change of basis and analyze the Artinian Gorenstein algebra $\mathbb{F}_p[x,y,z]/(x^m, y^n, z^k)$. In particular, we make use of the fact that for the values of $p,m,n,k$ in our range of interest, this algebra has the Weak Lefschetz property \cite{Lundqvist+Nicklasson.19}. We present this as an alternative computation to the explicit Gr\"obner computations in \Cref{sec:wlp}. Related computations of these lengths assuming some Lefschetz-type properties have previously been carried out in \cite[Section 2]{HanThesis}.

An interesting consequence of the formulas obtained in \Cref{thm:limitfsig} is that the values match with that of the \emph{normalized volume} of the corresponding KLT pair computed by Li in \cite{LiStabilityofSheaves}:
\begin{corollaryB*} \mbox{}[\Cref{cor:normalizedvolvsFsig}]
    Let $(X, \Delta) = (\Spec(\ZZ[x,y]), \divi (f^t))$, where $f= x^a \, y^b \, (x+y) ^c$ and $t \geq 0$ is a rational parameter. Then, we have
  \[\frac{\nvol(X_\mathbb{C}, \Delta_\mathbb{C})}{4} = \lim_{p \to \infty} s (X_p, \Delta_p) \]
    where $X_\mathbb{C}$ and $X_p$ denote the base changes of $X$ to $\mathbb{C}$ and to $\mathbb{F}_p$ and localized at $(x,y)$, respectively (and similarly for $\Delta$).
\end{corollaryB*}

The normalized volume introduced by Li \cite{LiMinimizingNormalizedVolumes} is an invariant of singularities over the complex numbers detecting the stability properties of KLT singularities. While defined very differently, the normalized volume and the $F$-signature satisfy many analogous properties (see \cite{Carvajal-rojas-Schwede-Tucker-Fundamentalgroups, MaPolstraSchwedeTuckerFsigunderbirationalmaps, Taylorinversionadjunctionfsignature}) and are expected to be closely related by reduction modulo~$p$. While the exact nature of their relation is not clear, \Cref{cor:normalizedvolvsFsig} provides further evidence towards their connection.

\subsection*{Acknowledgments} 
The authors thank Nasrin Altafi, Jenny Kenkel,  Yuchen Liu, Alexandra Seceleanu, and Karen E.~Smith for insightful conversations, as well as Michael Perlman whose undergraduate honors thesis contained a number of preliminary related investigations. We thank the referee for valuable comments and for pointing out the reference \cite{HanThesis} and other useful comments.

The authors thank the hospitality of the Simons-Laufer
Mathematical Sciences Institute (formerly the Mathematical Sciences Research Institute, and which
is supported by National Science Foundation Grant No. DMS-1928930 and the Alfred P. Sloan Foundation under grant G-2021-16778) where parts of the research
were carried out in Berkeley, California, during the Spring 2024 semester.

\section{Preliminaries}

In this section we recall a number of numerical Frobenius invariants of singularities. 
See \cite{BlickleSchwedeTuckerFSigPairs1} 
for the general definition of the $F$-signature functions in the pair and triples settings (especially Theorem 3.5); in this article, the following description in the hypersurface setting will suffice.

\begin{definition}[$F$-signature function of hypersurface pairs] \cite{BlickleSchwedeTuckerFSigPairs1, BlickleSchwedeTuckerFsigofPairs2} \label{def:Fsignaturefunction}
Let $\bK$ be a perfect field of characteristic $p > 0$, $R = \bK \llbracket  x_1, \ldots, x_n \rrbracket$ or the local ring $\bK [x_1, \dots , x_n]_{(x_1, \dots, x_n)}$, and $f \in \mathfrak{m} = (x_1, \dots, x_n)$ non-zero. For $t \in \mathbb{R}_{\geq 0}$, the $F$-signature of the pair $(R,f^t)$ is given by
\begin{align*}
    s(R, f^{t}) &= \lim_{e \to \infty} \frac{1}{p^{en}}\, \ell \left( \frac{R}{(x^{p^e}_{1}, \ldots, x^{p^e}_{n}):f^{\lceil tp^e \rceil}} \right) \\ &= 1 - \lim_{e \to \infty} \frac{1}{p^{en}}\, \ell \left( \frac{R}{(x^{p^e}_{1}, \ldots, x^{p^e}_{n},f^{\lceil tp^e \rceil})} \right).
\end{align*}
\end{definition}

By inspection, it is easy to see that $s(R,f^t)$ is non-increasing in $t \in \mathbb{R}_{> 0}$, and we have that $s(R,f^0) = 1$ and $s(R,f^1) = 0$.
In practice, computing $s(R,f^t)$ for arbitrary values of $t \in \mathbb{R}_{\geq 0}$ is difficult. However, if $t = \frac{a}{p^b}$ with $a,b \in \mathbb{Z}$, exploiting the flatness of Frobenius on $R$ gives that
\begin{equation}  \label{equation:Fsiglength} 
s(R, f^{\frac{a}{p^b}}) = \frac{1}{p^{bn}} \ell \left( \frac{R}{(x^{p^e}_{1}, \ldots, x^{p^e}_{n}):f^{a}} \right) = 1 -  \frac{1}{p^{bn}} \ell \left( \frac{R}{(x^{p^e}_{1}, \ldots, x^{p^e}_{n},f^{a})} \right).
\end{equation}
Importantly, one can show that the function $t \mapsto s(R,f^t)$ is continuous \cite[Theorem 3.3]{BlickleSchwedeTuckerFsigofPairs2} and even convex \cite[Theorem 3.9]{BlickleSchwedeTuckerFsigofPairs2}.

 A key result about the $F$-signature states that $s(R, f^t)$
 is positive if and only if $(R,f^t)$ is strongly $F$-regular \cite[Theorem 3.18]{BlickleSchwedeTuckerFSigPairs1} (cf.~\cite{AberbachLeuschke}). Further, the smallest $t$-intercept of the  $F$-signature function $t \mapsto s(R,f^t)$ is the $F$-pure threshold of $f$, defined below.

\begin{definition}[$F$-pure threshold, 
\cite{TakagiWatanabe}] \label{def:Fpurethreshold}
   Let $\bK$ be a perfect field of characteristic $p > 0$, $R = \bK \llbracket x_1, \ldots, x_n \rrbracket $ or the local ring  $\bK [x_1, \dots , x_n]_{(x_1, \dots, x_n)}$, and $f \in \mathfrak{m} = (x_1, \dots, x_n)$ non-zero. The $F$-pure theshold of $f$ (at $\fram$) is
   \[
   \fpt(f) = \sup \left\{ c = \frac{a}{p^e} \in \ZZ[\frac{1}{p}] \,\Big|\, f^a \not\in \fram^{[p^e]}\right\}.
   \]
\end{definition}
\noindent
While not immediate from the definition, the $F$-pure threshold is indeed a rational number \cite{BlickleMustataSmithFThresholdsOfHypersurfaces,BlickleMustataSmithDiscretenessAndRationalityOfFThresholds,BlickleSchwedeTakagiZhang}. 
It is closely related to another fundamental invariant, the log canonical threshold or $\lct$, whose origins lie in characteristic zero where it can be computed using a resolution of singularities. 
See \cite{BenitoFaberSmith} for an elementary introduction to these invariants and the relationship between them. 
Here we point out several of these connections in the context of reduction to characteristic $p>0$ for future use. 
\begin{proposition}[{\cite[Theorems 3.3 and 3.4]{MustataTakagiWatanabeFThresholdsAndBernsteinSato}}]
\label{fpt-lct-properties}
Let $f\in R=\mathbb Z[x_1,\ldots, x_n]$ be a non-zero polynomial contained in $\mathfrak m=(x_1,\ldots, x_n)$, so that by abuse of notation we can view $f\in \mathbb Z/p\mathbb Z[x_1,\ldots, x_n]_{\mathfrak{m}}$ or in $\mathbb Q[x_1,\ldots, x_n]_{\mathfrak{m}}$. Then for all $p\gg 0$, we have $\fpt(f) \leq \lct(f)$. Further, $\displaystyle\lim_{p\to\infty}\fpt(f) = \lct(f)$.
\end{proposition}

\begin{definition}[Limit $F$-signature function] \label{def:limitFsignaturefunction}
    Let $f \in R = \ZZ[x_1, \dots, x_n]$ be a non-zero polynomial contained in $\mathfrak{m} = (x_1, \dots, x_n)$. Note that for every prime $p>0$ such that $f \notin pR$, we can consider $f_p \in \ZZ/p \ZZ[x_1, \dots, x_n]$ which is a non-zero polynomial over $\bF _p$. Thus, we can consider the functions $\psi_p (t)$, defined as the $F$-signature of the pair $(\mathbb{F}_p[x_1, \dots, x_n]_\mathfrak{m}, f_p ^t)$, as a sequence of functions defined for all $p \gg 0$. We define the \emph{limit $F$-signature function} $\psi(t)$ as
    \[ \psi(t) = \lim_{p \to \infty} \psi_p (t), \]
    if it exists.
    \end{definition}

We note that the limit $F$-signature function is not known to exist outside of some special cases as in \cite{caminatashidelertuckerzermandiagonalhypersurfaces}, \cite{canton2012}.
However, as the following lemma shows, we do not need to compute the $F$-signature in every characteristic to compute the limit $F$-signature of polynomials over~$\ZZ$:

\begin{lemma} \label{lem:limitFsigofhypersurfaces} Let $f \in R = \ZZ[x_1, \dots, x_n]$ be a non-zero polynomial contained in $\mathfrak{m} = (x_1, \dots, x_n)$. For every prime $p$ such that $f \notin pR$, let $\psi_p (t)$ denote the $F$-signature of the pair $(\mathbb{F}_p[x_1, \dots, x_n]_\mathfrak{m}, f^t)$. Suppose we have a continuous, non-increasing function $\psi(t)$ on the interval $[0,1]$ with $\psi(0) = 1$ and $\psi(1) = 0$ such that for each rational number $u \geq 0$, there is sequence $r_p (u)$ of non-negative real numbers such that:
\begin{itemize}
    \item $r_p (u) \leq u$ for each $p \gg 0$ and $r_p (u) \to u$ as $p \to \infty$, and
    \item $\psi_p (r_p (u))$ converges to $\psi(u)$ as $p \to \infty$.
\end{itemize}
 Then, the sequence of functions $\psi_p$ converges uniformly to the function $\psi$ as $p \to \infty$. In particular, if $\psi_p (u) \to \psi(u) $ as $p \to \infty$ for each rational number $u$, then $\psi_p \to \psi$ uniformly.
\end{lemma}

\begin{proof}
    First we show point-wise convergence: for each $t \in [0,1]$, the sequence $\psi_p (t)$ converges to $\psi(t)$. This is clear for $t = 0,1$, so we assume that $ t$ is any real number in $(0,1)$. Choose $u_1, u_2 \in (0,1) \cap \mathbb{Q}$ such that $u_1 < t < u_2$. Using our assumption, we know that the sequences $\psi_p \left(r_p (u_1)\right)$ and $\psi_p \left(r_p(u_2)\right)$ converge to $\psi(u_1)$ and $\psi(u_2)$ respectively (as $p \to \infty$). Since each function $\psi_p$ is non-increasing, for each $p \gg 0$ we have
    \[ \psi_p\left(r_p(u_2)\right) \leq \psi_p (t) \leq \psi_p \left(r_p (u_1)\right). \]
    This holds since $u_2 >t$ implies we have $r_p (u_2) >  t$ for $p \gg 0$. By considering the limit $p \gg 0$ in the above inequality, we have
    \begin{equation} \label{eqn:limsupinf} \psi(u_2) \leq \liminf_{p \to \infty} \psi_p (t)  \leq \limsup_{p \to \infty} \psi_p (t) \leq \psi(u_1). \end{equation}
    Furthermore, taking rational numbers $u_1 \to t$ and $u_2 \to t$ and using the continuity of $\psi(t)$, we must have $\limsup \psi_p (t) = \liminf \psi_p(t)  = \psi(t)$. This shows point-wise convergence of $\psi_p$ to~$\psi$.

    The uniform convergence follows from point-wise convergence since each $\psi_p (t)$ is continuous and convex. See \cite[Corollary 1.3.8]{niculescuperssonconvexanalysis}.
\end{proof}

\section{\texorpdfstring{Lengths and Gr\"obner bases when $f=x+y$}{Lengths and Groebner bases when f=x+y}}
\label{sec:simplified-f-case}

We will start by doing these computations in a simplified setting where $f=x+y$ is a polynomial in $\bK[x,y]$ for $\bK$ a field. More specifically, our goal in this section is to find a Gr\"obner basis for the ideal
\[
I_{k,m,n} =(x^m, y^n, (x+y)^k),
\]
and then use this to find the length 
\[
\ell_{k,m,n} = \ell(\bK[x,y]/I_{k,m,n}).
\]
We fix a monomial order with $x>y$. 
In the first subsection, we will do a detailed Gr\"obner basis computation when $m\leq k$, and use this to find the lengths. Then in \Cref{sec:k-less-than-m} we will use the old Gr\"obner basis to find a new one when $m>k$. Finally in \Cref{sec:simplified-lengths} we will compute the lengths $\ell_{k,m,n}$.

\subsection{Case \texorpdfstring{$m \leq k$}{m at most k}} \label{section:mlessthank} 

Our goal is to find a Gr\"obner basis for the ideal $I_{k,m,n}$ when $m\leq k$. Inspired by Buchberger's algorithm, we first compute some $\spair$-pairs. We will then give a nicer description of the resulting polynomials in \Cref{prop-recursion1}, and we will give a complete Gr\"obner basis using some of these polynomials in \Cref{prop-Grobnerbasis1}.

We will write $T^x_m(g)$ to mean the ``truncation of $g$ at $x^m$'', i.e., we throw away all terms whose $x$-degree is~$\geq m$. Similarly, $T^y_n(g)$ means throwing away all terms whose $y$-degree is~$\geq n$.

\begin{definition}\label{def-recursion1}
Define \[f_1 =T_m^x f^k= \sum_{j=0}^{m-1} \binom{k}{j} x^{j} y^{k-j}.\]
\[f_2 = 
x f_1 - \binom{k}{m-1} y^{k-m+1} x^m = \sum_{j=0}^{m-2} \binom{k}{j} x^{j+1}y^{k-j}.\]
Now for $t<m$, recursively define
\begin{equation*}
    \begin{array}{rcl}
        f_{2t+1} &=& 
        \displaystyle yf_{2t-1} - \left(\displaystyle\frac{k+2t-m}{m-t}\right) f_{2t}.\\
         f_{2t+2} &=& 
         \left(\displaystyle \frac{k+2t-m+1}{k+t-m+1}\right) x f_{2t+1} - \left(\displaystyle\frac{t(k+t)}{(k+t-m+1)(m-t)}\right) y f_{2t}.
    \end{array}
\end{equation*}
\end{definition}

We use the convention that the empty product is 1. The following proposition solves the recursion. 

\begin{proposition}\label{prop-recursion1}
The recursion for $0 \leq t <m$ (or $1\leq t < m$ for the even case) is solved by the following functions:
\[
\begin{array}{rcl}
     f_{2t+1} & =& \displaystyle \sum_{j=0}^{m-1-t} \frac{\prod_{\ell=1}^{t}(m-j-\ell)}{\prod_{\ell=1}^{t}(m-\ell)}\binom{k+t}{j}x^{j}y^{k-j+t}.\\
     f_{2t} & = &\displaystyle \sum_{j=0}^{m-1-t} \frac{\prod_{\ell=2}^{t}(m-j-\ell)}{\prod_{\ell=1}^{t-1}(m-\ell)}\binom{k+t-1}{j} x^{j+1}y^{k-j+t-1}.
\end{array}
 \]
\end{proposition}

\begin{proof}
 When $t=0$, $f_{2t+1}$ simplifies to $f_1$. Similarly, when $t=1$, $f_{2t}$ simplifies to $f_2$. We now check the recursions. Substituting the expressions for $f_{2t}$ and $f_{2t+1}$ into the equation
\[f_{2t+2} = \frac{k+2t-m+1}{k+t-m+1}x f_{2t+1} - \frac{t(k+t)}{(k+t-m+1)(m-t)}y f_{2t}, \]
we see that 
 \begin{equation*}
 \begin{array}{rcl} \displaystyle
     f_{2t+2} &=& \displaystyle \left(\frac{k+2t-m+1}{k+t-m+1}\right) x 
     \left(\sum_{j=0}^{m-1-t} \frac{\prod_{\ell=1}^{t}(m-j-\ell)}{\prod_{\ell=1}^{t}(m-\ell)} \binom{k+t}{j} x^{j}y^{k-j+t}\right)   \\ \displaystyle
     &-& \displaystyle \left(\frac{t(k+t)}{(k+t-m+1)(m-t)}\right) y 
     \left(\sum_{j=0}^{m-1-t} \frac{\prod_{\ell=2}^{t}(m-j-\ell) }{ \prod_{\ell=1}^{t-1}(m-\ell)} \binom{k+t-1}{j} x^{j+1}y^{k-j+t-1}\right) \\ 
\end{array}
\end{equation*}
and so the coefficient of $x^{j+1} y^{k-j+t}$ is
\begin{align*}
&\left(\frac{k+2t-m+1}{k+t-m+1}\right) \left(\frac{\prod_{\ell=1}^{t}(m-j-\ell)}{\prod_{\ell=1}^{t}(m-\ell)}\right) \binom{k+t}{j}   \\
&\qquad\qquad\qquad - 
     \left(\frac{t(k+t)}{(k+t-m+1)(m-t)}\right)  \left(\frac{\prod_{\ell=2}^{t}(m-j-\ell) }{ \prod_{\ell=1}^{t-1}(m-\ell)}\right) \binom{k+t-1}{j}.
\end{align*}

Factoring out the common product of fractions and a binomial coefficient  
\[\left(\frac{\prod_{\ell=2}^{t}(m-j-\ell)}{\prod_{\ell=1}^{t-1}(m-\ell)}\right) \binom{k+t-1}{j},\] 
we are left with
\[    \left(\frac{(k+2t-m+1)(m-j-1) }{ (k+t-m+1)(m-t)}\right) \left(\frac{k+t}{k+t-j}\right)
    - \left(\frac{t(k+t)}{(k+t-m+1)(m-t)}\right),\]
which simplifies to     
     \[\frac{(m-j-t-1)(k+t)}{(m-t) (k+t-j)}.\]

We can now recombine these extra factors into the binomial coefficient and the fraction of products out front to get an $x^{j+1}y^{k-j+t}$ coefficient of
\[
\left(\frac{\prod_{\ell=2}^{t+1}(m-j-\ell)}{\prod_{\ell=1}^{t}(m-\ell)}\right)\binom{k+t}{j}.
\]
Thus
\begin{align*}
f_{2t+2} &= \sum_{j=0}^{m-1-t} 
\left(\frac{\prod_{\ell=2}^{t+1}(m-j-\ell)}{\prod_{\ell=1}^{t}(m-\ell)}\right) \binom{k+t}{j} x^{j+1}y^{k-j+t} \\
    &=\sum_{j=0}^{m-2-t} \left(\frac{\prod_{\ell=2}^{t+1}(m-j-\ell)}{\prod_{\ell=1}^{t}(m-\ell)}\right) \binom{k+t}{j} x^{j+1}y^{k-j+t} ,
\end{align*}
since when $j=m-1-t$, the $\ell=t+1$ term of the product $\prod_{\ell=2}^{t+1}(m-j-\ell)$ is zero.
Hence, the recursion is satisfied.

Similarly, substituting the expressions for $f_{2t-1}$ and $f_{2t}$ into
\[ f_{2t+1} = y f_{2t-1} - \left(\frac{k+2t-m}{m-t}\right) f_{2t},   \]
we obtain that  
{ \everymath={\displaystyle}
\begin{equation*}
    \begin{array}{rcl}
         f_{2t+1}&=& y \sum_{j=0}^{m-t} \frac{\prod_{\ell=1}^{t-1}(m-j-\ell)}{\prod_{\ell=1}^{t-1}(m-\ell)}\binom{k+t-1}{j}x^{j}y^{k-j+t-1} \vspace{1em}\\ 
        && - \,\,\, \left(\frac{k+2t-m}{m-t} \right) \sum_{j=0}^{m-1-t} \frac{\prod_{\ell=2}^{t}(m-j-\ell)}{\prod_{\ell=1}^{t-1}(m-\ell)} \binom{k+t-1}{j} x^{j+1}y^{k-j+t-1} \vspace{1em} \\
         &=& \sum_{j=0}^{m-t} \frac{\prod_{\ell=1}^{t-1}(m-j-\ell)}{\prod_{\ell=1}^{t-1}(m-\ell)}\binom{k+t-1}{j}x^{j}y^{k-j+t}\vspace{1em} \\
         && - \,\,\, \left(\frac{k+2t-m}{m-t}\right) \sum_{j=1}^{m-t} \left( \frac{\prod_{\ell=1}^{t-1}(m-j-\ell)}{\prod_{\ell=1}^{t-1}(m-\ell)} \right) \binom{k+t-1}{j-1} x^{j}y^{k-j+t} \vspace{1em} 
    \end{array}
\end{equation*}
}

When $j=0$, the coefficient of $y^{k-t}$ is already of the form 
\[
\frac{\prod_{\ell=1}^{t-1}(m-\ell)}{\prod_{\ell=1}^{t-1}(m-\ell)}\binom{k+t-1}{0} = 1 = \frac{\prod_{\ell=1}^{t}(m-\ell)}{\prod_{\ell=1}^{t}(m-\ell)}\binom{k+t}{0}.
\]
Otherwise, for $1\leq j \leq m-t$, we have an $x^j y^{k-j+t}$ coefficient of \begin{align*}
&\left( \frac{\prod_{\ell=1}^{t-1}(m-j-\ell)}{\prod_{\ell=1}^{t-1}(m-\ell)}\right) \binom{k+t-1}{j} 
    - \left(\frac{k+2t-m}{m-t}\right)\left( \frac{\prod_{\ell=1}^{t-1}(m-j-\ell)}{\prod_{\ell=1}^{t-1}(m-\ell)} \right) \binom{k+t-1}{j-1}\\
&\, =\left(\frac{\prod_{\ell=1}^{t-1}(m-j-\ell)}{\prod_{\ell=1}^{t}(m-\ell)}\right) \left( \frac{(k+t-1)!}{j!(k+t-j)!} \right) \Big((k+t-j)(m-t)- j(k+2t-m)\Big)
\end{align*}

Since $(k+t-j)(m-t)- j(k+2t-m) = (k+t)(m-j-t)$ and since the coefficient again simplifies to zero when $j=m-t$, we obtain that 
\[f_{2t+1} = \sum_{j=0}^{m-t-1} \frac{\prod_{\ell=1}^{t}(m-j-\ell)}{\prod_{\ell=1}^{t}(m-\ell)} \binom{k+t}{j} x^{j}y^{k-j+t}. \]
This concludes the induction.
\end{proof}

Now we have the following useful lemma.

\begin{lemma}\label{lem-notdivide}
Let $m\leq k$ and  $m+k\leq p$ for a prime number $p$, and let $1 \leq s < 2m$ be an integer.  Let $f_s$ be one of the functions recursively defined in Definition~\ref{def-recursion1}. Then $p$ does not divide any of the coefficients of $f_s$.

In particular, in both characteristic zero and characteristic $p$, the $f_{s}$ polynomials have leading terms
\begin{align*}
\LT(f_{2t+1}) &= \frac{\prod_{\ell=1}^{t}(t-\ell+1)}{\prod_{\ell=1}^{t}(m-\ell)}\binom{k+t}{m-t-1}x^{m-t-1}y^{k-m+2t+1} 
    = \frac{\binom{k+t}{m-t-1}}{\binom{m-1}{t}} x^{m-1-t}y^{k+2t+1 - m}, \\
\LT(f_{2t}) &= \frac{\prod_{\ell=2}^{t} (t-\ell+1)}{\prod_{\ell=1}^{t-1}(m-\ell)} \binom{k+t-1}{m-t-1}x^{m-t}y^{k-m+2t} 
    = \frac{\binom{k+t-1}{m-t-1}}{\binom{m-1}{t-1}} x^{m-t} y^{k+2t -m}.
\end{align*}
\end{lemma}

\begin{proof}

Write either $s=2t$ or $s=2t+1$, depending on parity. We may assume $s\geq 3$ (and thus $m,t\geq 1$) since the result clearly holds in the base cases. 
Recall that the recursion coefficients appearing in \Cref{def-recursion1} are $\frac{k+2t-m}{m-t}$, $\frac{k+2t-m+1}{k+t-m+1}$, and $\frac{t(k+t)}{(k+t-m+1)(m-t)}$. Since $t<m\leq k$ and $1\leq t \leq k+t<k+m\leq p$, we have that
\begin{align*}
1\leq k+t-m+1 \leq k+2t-m+1 < k+t+1 \leq k+m \leq p
\end{align*}
and similarly that $1\leq k+2t-m <p$. Hence, all the coefficients appearing in the definitions of the recursion are nonzero and less than $p$. Therefore, the functions $f_s$ are well-defined in characteristic $p$. 

If $s=2t+1$, then for $0 \leq j \leq m-t-1$, the $x^jy^{k-j+t}$ coefficient of $f_s$ is 
\[
\frac{\prod_{\ell=1}^{t}(m-j-\ell)}{\prod_{\ell=1}^{t}(m-\ell)} \binom{k+t}{j}.
\]
This coefficient is non-zero and well-defined since 
\[
1\leq m-j-t\leq m-t, \ \ 1\leq k+t,\  \textrm{ and } \ j\leq m-t-1<m\leq k<k+t.
\]
This coefficient is not divisible by $p$ because $m-j-\ell, m-\ell < m \leq p$ and $k+t<k+m<p$.

If $s=2t$, then since $s\geq 3$ we may in fact assume $t\geq 2$. For $0 \leq j \leq m-t-1$, the  $x^{j+1}y^{k-j+t-1}$ coefficient of $f_s$ is
\[
\frac{\prod_{\ell=2}^{t}(m-j-\ell)}{\prod_{\ell=1}^{t-1}(m-\ell)}\binom{k+t-1}{j}.
\]
This coefficient is non-zero and well-defined since 
\[
1\leq m-j-t\leq m-t+1,\ \ 1\leq k+t-1,\ \textrm{ and } \ j\leq m-t-1 <m\leq k+t-1.
\]
This coefficient is not divisible by $p$ because
$m-\ell, m-j-\ell < m \leq p$ and $k+t-1 < k+m < p$.
This concludes the proof of the lemma.
\end{proof}

\begin{proposition}\label{prop-Grobnerbasis1}
Assume that $m \leq k $, that  $I_{k,m,n}=(x^m,y^n,(x+y)^k)$, and that the $f_{2t+1}$ are as in \Cref{prop-recursion1}. Fix a monomial order with $x>y$. 
\begin{enumerate}
\item If the field $\bK$ has characteristic $0$, then the elements 
\[x^m, y^n, f_{2t+1} \ \ \mbox{for} \ \ 0 \leq t < m\] form a Gr\"{o}bner basis of the ideal $I_{k,m,n}$.
\item If the field $\bK$ has characteristic $p$ and $m+k\leq p$, then the elements 
\[x^m, y^n, f_{2t+1} \ \ \mbox{for} \ \ 0 \leq t < m\] form a Gr\"{o}bner basis of the ideal $I_{k,m,n}$. 
\end{enumerate}
\end{proposition}

\begin{remark}
In fact, since $y^n$ is in the proposed Gr\"obner basis, this means that taking $x^m$, $y^n$, $T^y_nf_{2t+1}$ should also be a Gr\"obner basis. This truncation is non-zero whenever the smallest $y$-degree appearing in a term is $<n$, i.e., whenever  $2t+k+1-m  <n$ or equivalently $t<\frac{n+m-k-1}{2}$. Thus, a Gr\"obner basis is precisely the polynomials
\[
x^m, y^n, T^y_n f_{2t+1} \ \ \mbox{for}\ \ 0\leq t < \min\left(m, \frac{n+m-k-1}{2}\right).
\]
\end{remark}

\begin{proof}[Proof of  \Cref{prop-Grobnerbasis1}]
Given an element $g \in I_{k,m,n}$, we will show that the leading term is divisible by $x^m$, $y^n$, or the leading term of one of the $T_n^y f_{2t+1}$. Any element of $I_{k,m,n}$ can be expressed as 
\[g = g_1 x^m + g_2 y^n + h f^k,\] 
where $g_1, g_2, h \in \bK[x,y]$. 
If the leading term of $g$ is divisible by $x^m$ or $y^n$, we are done. We may therefore assume that the leading term $x^{\alpha}y^{\beta}$ of $g$ has $\alpha<m$ and $\beta < n$. Since every term in $g_1 x^m + g_2 y^n$ is divisible by $x^m$ or $y^n$, the term $x^{\alpha}y^{\beta}$ must be the leading term of $h f^k$ modulo the terms that are divisible by $x^m$ or $y^n$. 

Since the ideal $I_{k,m,n}$ is homogeneous with respect to the usual grading on $\bK[x,y]$, it is sufficient to consider the case where $g_1$, $g_2$, and $h$ are all homogeneous polynomials, and so we can write our non-zero $h$ as $h(x,y) = \sum_{\ell = 0}^{d} a_\ell x^{\ell}y^{d - \ell}$. We first compute
\begin{align*}
hf^k &= \left(\sum_{\ell = 0}^{d} a_\ell x^{\ell}y^{d-\ell}\right) \left(\sum_{j= 0} ^{k} \binom{k}{j} x^{j} y^{k-j}\right) \\
&= \sum_{j=0}^{k+d} \left( \sum_{\ell = \max(0, j-k)}^{\min(d,j)} a_\ell \binom{k}{j-\ell} \right) x^{j} y^{k+d-j}.
\end{align*}
Note that we can express this information succinctly in a (labelled) matrix:
\[
M (f^k, m, n, d): = \kbordermatrix{
& a_0 & a_1 & a_2 & a_3 & \cdots & a_{d} \\
j=0 & \binom{k}{0} & 0 & 0 & 0 & \cdots & 0 \\
j=1 & \binom{k}{1} & \binom{k}{0} & 0 & 0 & \cdots & 0 \\
j=2 & \binom{k}{2} & \binom{k}{1} & \binom{k}{0} & 0 & \cdots & 0 \\
\vdots & \vdots \\
j=k & \binom{k}{k} & \binom{k}{k-1} & \cdots \\
j=k+1 & 0 & \binom{k}{k} \\
\vdots & \vdots & \vdots \\
j=k+d & 0 & 0 & 0 & 0 & \cdots & \binom{k}{k}
}
\]
where the row labelled $``j"$ corresponds to the term where the $x$-degree equals $j$. The coefficient of the $x^{j}$ term is the product of the $j$-row with the column vector $\left [ a_0, \dots, a_d \right ] ^\text{T}$.
Suppose the leading term of $T^x_m (h f^k)$ is $x^{\alpha} y^{\beta}$ for non-negative integers $\alpha$ and $\beta$  such that $\alpha<m$. 
Let $t=m-1-\alpha$, so that
\[ \alpha =  m-1-t \]
and $0 \leq t < m $. Note that by homogeneity, 
\[ \beta = k+d - \alpha = k + d + t +1-m.  \]

We will show that this choice of $t$ ensures that this leading term $x^{\alpha} y^\beta$ is a multiple of $\LT( f_{2t+1})$. It suffices to show $\beta \geq k+2t +1-m$, so that $x^{\alpha}y^{\beta}$ is divisible by $x^{m-1-t} y^{k+2t+1-m}$ (our assumed constraints and \Cref{lem-notdivide} ensure that this is indeed the correct initial monomial). 
If we knew that $d\geq t$, then we would have 
\[\beta  = k + d  +t+1-m \geq  k + 2t + 1 - m\] 
as desired. Thus, all that remains is to prove the following claim.

\textbf{Claim:}  $d \geq t$. 

\textbf{Proof of claim:} We will obtain a contradiction by instead assuming $ d+1 \leq t$. 
Consider the terms appearing in $T^x_m (h f^k)$. By assumption, the coefficient of $x^{\alpha} y^{\beta}$ is nonzero. Since this is the leading term, the coefficients of $x^{\alpha + \iota } y^{\beta- \iota }$ must vanish for $1 \leq \iota \leq t$. By assumption all the exponents appearing in this list are less than~$m$. 
Consider these constraints on the coefficients of the $x^{\alpha+\iota}y^{\beta-\iota}$ as a sequence of $t$ equations with the $a_\ell$ as our unknowns. If $t>d+k$, we can simply consider the extra constraints as ``trivial'' constraints.
Then, we make the following observations about the matrix representing this system of equations:
\begin{enumerate}
    \item The ``bottom" row corresponds to $\iota = t$, i.e., the $j = m-1$ row of $M(f^k, m,n,d)$.
    \item The ``top" row corresponds to $\iota=1$, i.e.,  the $j = i  - t+1=m -t$ row.
  
    \item The number of rows is equal to $t$, which is at least $d+1$ by assumption.
\end{enumerate}
Now let $D$ be the top $(d+1)\times (d+1)$ submatrix of these constraints, so that
\[
D = \kbordermatrix{
& a_0 & a_1 &  \cdots & a_{d} \\
j=m -t & \binom{k}{m-t} & \binom{k}{m-t-1} & \cdots & \binom{k}{m-t-d} \\
j=m -t+1 & \binom{k}{m -t+1} & \binom{k}{m-t} & \cdots & \binom{k}{m-t-d+1} \\
\vdots & \vdots & \vdots & \ddots & \vdots \\
j=m-t+d  & \binom{k}{m-t+d} & \binom{k}{m-t+d+1} & \cdots & \binom{k}{m-t}
},
\]
or in other words, $D$ is the matrix with entries
\[D_{i,j} = \binom{k}{m-t-i+j}\]
for $0\leq i,j\leq d$.
Then by applying \Cref{lem-generalizeddeterminant} to the matrix $D=D(k, m-t, d+t-m)$, we see that $D$ is always invertible in characteristic $0$, and is also invertible modulo $p$ since
\[k+(m-t)+(d+t-m) = k+d<k+t<k+(p-k)=p.\]
Thus the only way for all of these coefficients of $x^{\alpha+\iota}y^{\beta-\iota}$ to be zero is for all of the $a_\ell$ to be zero. But this contradicts the fact that $h\neq 0$, and thus $d\geq t$, proving the claim.
\end{proof}

\begin{lemma}\label{lem-generalizeddeterminant} 
Let $k,a,v$ be three integers such that $0 \leq a \leq k$ and $a+v \geq 0$.
Let $D(k,a,v)$ be the $(a+v+1) \times (a+v+1)$ matrix with entries $D(k, a, v)_{i,j} = \binom{k}{a-i+j}$ for $0 \leq i,j \leq a+v$, where we define $\binom{k}{\ell}$ to be zero when either $\ell<0$ or $\ell>k$. Then 
\begin{equation} \label{determinant-formula}
     \det D(k,a,v) = \frac{ \binom{k}{a} \binom{k+1}{a+1} \dots \binom{k+a+v}{2a+v}}{\binom{k+v}{a+v} \binom{k+v-1}{a+v-1} \dots \binom{k-a}{0}}. \end{equation} 
In particular, if $k+a+v < p$, then $ \det D(k,a,v) \not= 0$ modulo $p$.
\end{lemma}
\begin{proof}
    This determinant is computed in \cite[Lemma~4.1]{LiStrouseBinomialDeterminant} and \cite[Section 2]{HanThesis}. Note that all the binomial coefficients appearing in the formula for the determinant in \Cref{determinant-formula} are non-zero integers. Moreover, if $k+a+v < p$, then every factor in the formula in \Cref{determinant-formula} is non-zero modulo $p$, so the determinant is non-zero modulo $p$.
\end{proof}

\subsection{Case \texorpdfstring{$k < m$}{k less than m}} 
\label{sec:k-less-than-m}
Now we consider the complementary case, where $k<m$. This time, we will find a Gr\"obner basis for the ideal $I_{k,m,n}$ by using a suitable change of variables on our old Gr\"obner basis. 

Recall from \Cref{section:mlessthank} that when $m \leq k$, the polynomials $x^m$, $y^n$ and $f_{2t+1}$ for $0\leq t < m$ defined as in \Cref{prop-Grobnerbasis1} form a Gr\"obner basis for the ideal $I_{k,m,n}= (x^m, y^n, (x+y)^k )$, provided $m + k \leq p$ and the base field $\bK$ has characteristic 0 or $p$.

\begin{definition}
Let $\phi$ be the change of variables $x\mapsto x+y$, $y\mapsto -y$. Now define $g_t$ to be the result of applying $\phi$ to  $f_{2t+1}$, so that
\begin{align*}
g_t &= \sum_{j=0}^{k-1-t} \left( \frac{\prod_{\ell=1}^t (k-j-\ell)}{\prod_{\ell=1}^t (k-\ell)}  \right)
\binom{m+t}{j} (x+y)^j(-y)^{m-j+t}\\
    &=\sum_{i=0}^{k-1-t} \left(\sum_{j=i}^{k-1-t} (-1)^{m-j-t} \left(\frac{\prod_{\ell=1}^t(k-j-\ell)}{\prod_{\ell=1}^t(k-\ell)}\right) \binom{m+t}{j}\binom{j}{i} \right) x^i y^{m-i+t}.
\end{align*}
\end{definition}

\begin{proposition}\label{prop-Grobnerbasis2}
Assume that $ k < m$, and let $ I_{k,m,n} = ( x^m, y^n , (x+y)^k )$. Fix a monomial order with $x>y$.
\begin{enumerate}
\item If the field $\bK$ has characteristic $0$, then the elements
\[
f^k, y^n, g_t \ \ \mbox{for} \ \ 0\leq  t <k\]
give a Gr\"{o}bner basis of the ideal $I_{k,m,n}$.

\item If the field $\bK$ has characteristic $p$ and $m+k \leq p$, then the elements
\[f^k, y^n, g_t \ \ \mbox{for} \ \ 0 \leq t < k\] 
form a Gr\"{o}bner basis of the ideal $I_{k,m,n}$.
\end{enumerate}
\end{proposition}

\begin{proof}
Note that applying the change of variables $\phi:$ $x\mapsto x+y$, $y\mapsto -y$ to the ideal $I_{m,k,n}$ yields our desired ideal $I_{k,m,n}$. Since $k<m$, by \Cref{prop-Grobnerbasis1}
we have that
\[
x^k,y^n, f_{2t+1} \ \textrm{ for } 0\leq t < k
\]
give a Gr\"obner basis of the ideal $I_{m,k,n} = ( x^k, y^n, (x+y)^m)$.

Our change of variables in particular ensures $\LT(\phi(h)) = \LT(h)$ for every polynomial $h$. Thus $\phi$ does not affect the initial term ideal, i.e., $I_{k,m,n}$ and $I_{m,k,n}$ have the same initial term ideal. Likewise, $\phi$ does not affect the leading terms of the Gr\"obner basis. Since our new proposed basis of $f^k, y^n, g_i$ are all clearly in $I_{k,m,n}=\phi(I_{m,k,n})$, we thus have a set of elements in $I_{k,m,n}$ whose initial terms generate the initial term ideal, as desired.
\end{proof}

\subsection{The lengths}
\label{sec:simplified-lengths}

We can now use the Gr\"obner bases from \Cref{prop-Grobnerbasis1} and \Cref{prop-Grobnerbasis2} to get a unified description of the length of the ideal $I_{k,m,n}$.
To do this, we will first compute the lengths of  a family of monomial ideals which appear in both the $m\leq k$ and $m>k$ cases.

\begin{lemma}
\label{monomial-len-lemma}
Assume $\alpha \geq \beta$ and $\eta$ are all integers such that $\alpha + \beta \leq p$ for $p$ a prime. Consider the monomial ideal
\[
J = ( x^\beta,\ y^\eta) + \left( x^{\beta -1 - t  } y^{\alpha-\beta + 1+2t} \,\, | \,\, 0\leq t < \beta\right)
\]
in $\bK[x,y]$, where $\bK$ is a field of either characteristic zero or characteristic $p$.
\begin{enumerate}
\item If $\beta + \eta \leq \alpha$, then
\[
\ell(\bK[x,y]/J) = \beta \eta.
\]

\item If and 
$\alpha+\beta \leq \eta$, then
\[
\ell(\bK[x,y]/J) = \beta \alpha.
\]

\item If $\beta + \eta > \alpha$ and $\alpha+\beta >\eta$, then
\[
\ell(\bK[x,y]/J) = \beta\eta - \left\lfloor \frac{(\beta+\eta-\alpha)^2}{4}\right\rfloor.
\]
\end{enumerate}
\end{lemma}
\begin{proof}
We first break into cases and compute, then at the end we will re-combine based on some overlap.

\underline{\textbf{Case $\beta + \eta \leq \alpha+1$}}:
When $\beta + \eta \leq \alpha+1$, then for every $0\leq t < \beta$ we have 
\[
\alpha-\beta+1+2t \geq \alpha-\beta+1 \geq \eta.
\]
So in fact, $J = (x^\beta, y^\eta)$ which has the desired length.

\underline{\textbf{Case $\alpha +\beta \leq \eta$ \& $\beta + \eta > \alpha+1$}}: If $\beta+\eta >\alpha+1$, we go back to considering the other monomials in the ideal $J$. The largest $y$ degree occurs when $t=\beta-1$, for the monomial $x^{0}y^{\alpha+\beta-1}$. In particular, when $\alpha+\beta \leq \eta$, every one of the $x^{\beta-t-1} y^{2t+\alpha-\beta+1}$ is necessary. By ``counting boxes'', we see that
\begin{align*}
\ell(\bK[x,y]/J) &= (\alpha-\beta+1)\beta + \sum_{t=1}^{\beta-1} 2(\beta-t) \\
    &= (\alpha-\beta+1)\beta + 2\beta(\beta-1) - (\beta-1)\beta \\
    &= \beta\alpha.
\end{align*}

\underline{\textbf{Case $\alpha+\beta > \eta$ \& $\beta+\eta > \alpha+1$}}:
In this case, some of the monomials will become redundant with $y^\eta$. Let $t_0$ be the first index such that $x^{\beta-t_0-1}y^{2t_0+\alpha-\beta+1}$ is redundant, i.e., such that $2t_0+\alpha-\beta+1 \geq \eta$. This means $t_0 = \lfloor \frac{\beta+\eta-\alpha}{2}\rfloor$. We now split into cases based on the parity of $\beta + \eta - \alpha$. 

First, suppose that $2t_0+\alpha-\beta+1 = \eta$. In this case, our set of monomials not in the ideal is depicted in \Cref{fig:one}.
\begin{figure}\begin{center}
    \includegraphics{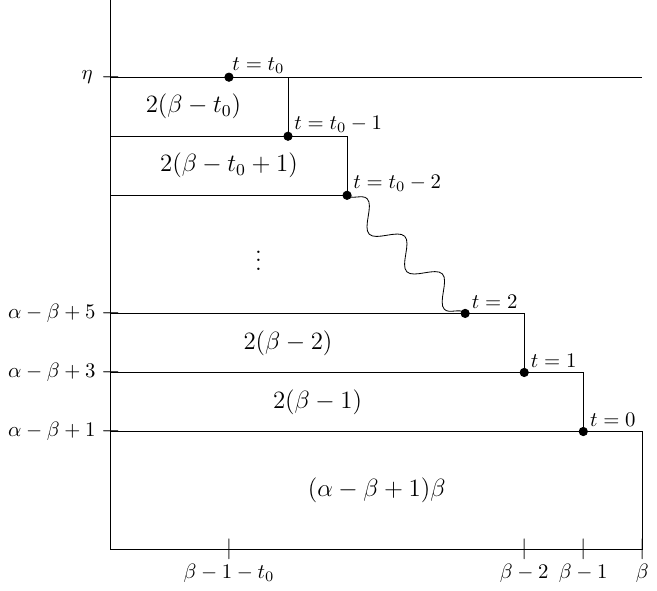}
    \end{center}
    \caption{Case: There exists $t_0$ with $\alpha-\beta+1+2t_0 = \eta$.}
    \label{fig:one}
    \end{figure}
In particular, we get
\[
\ell(\bK[x,y]/J) = (\alpha-\beta+1)\beta + \sum_{t=1}^{t_0} 2(\beta-t) = \beta(\alpha-\beta+1) + 2\beta t_0 -t_0(t_0+1).
\]
Since $\alpha-\beta+1=\eta-2t_0$ and $\beta+\eta-\alpha$ is odd, this can be simplified to 
\[
\ell(\bK[x,y]/J) = \beta(\eta-2t_0) + 2\beta t_0-t_0(t_0+1) = \beta\eta - t_0(t_0+1) = \beta\eta - \left\lfloor\frac{(\beta+\eta-\alpha)^2}{4}\right\rfloor
\]

In the other case, then $2t_0+\alpha-\beta =\eta$ and our set of monomials not in the ideal is depicted in \Cref{figure:two}.

\begin{figure} 
    \begin{center}
    \includegraphics{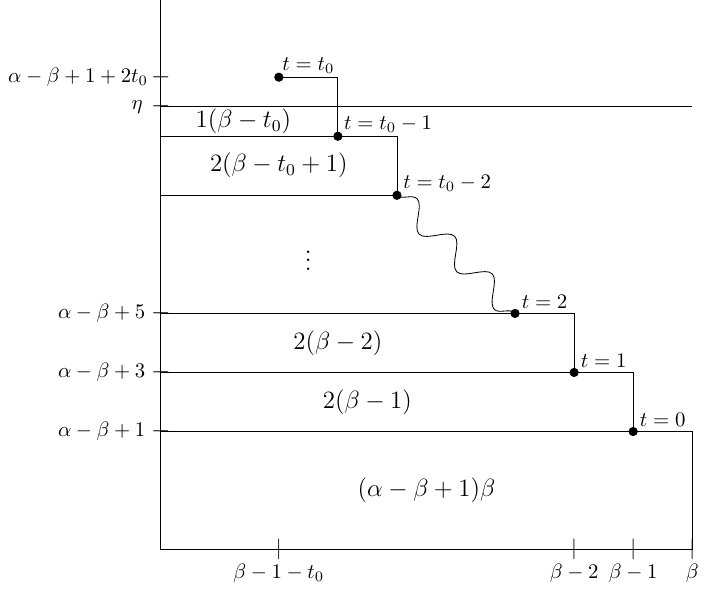}
    \end{center}
    \caption{Case: There exists $t_0$ with $\alpha-\beta+2t_0 = \eta$.}
    \label{figure:two}
    \end{figure}
In particular, we get
\begin{align*}
\ell(\bK[x,y]/J) &= (\alpha-\beta+1)\beta+\left(\sum_{t=1}^{t_0-1} 2(\beta-t)\right) + (\beta-t_0) \\
&= \beta(\alpha-\beta+1)+2\beta(t_0-1)-(t_0-1)t_0 + (\beta-t_0) \\
&= \beta(\alpha-\beta) +2\beta t_0-t_0^2.
\end{align*}
We can again simplify using the fact that $\alpha-\beta = \eta-2t_0$ and $\beta+\eta-\alpha$ is even, to get
\[
\ell(\bK[x,y]/J) = \beta(\eta-2t_0)+2\beta t_0-t_0^2 = \beta\eta - t_0^2 = \beta\eta - \left\lfloor\frac{(\beta+\eta-\alpha)^2}{4}\right\rfloor.
\]

To simplify the statement, observe that if $\beta+\eta = \alpha+1$, then $\beta\eta - \left\lfloor\frac{(\beta+\eta-\alpha)^2}{4}\right\rfloor = \beta \eta$, so we re-group that into case (c).

Finally, we note that there is overlap between cases (a) and (b) in the statement of the lemma, but the formula is the same---if $\alpha+\beta\leq \eta$ and $\beta+\eta\leq \alpha$, this forces $2\beta \leq 1$. Since $\beta$ is an integer, this means $\beta =0$. 
\end{proof}

\begin{theorem}[\mbox{\textit{cf.} \cite[Theorem 2.3]{HanThesis}}]
\label{unified-simpler-lengths}
Let $m,n,k$ be non-negative integers such that $\min(m+k,m+n, n+k)\leq p$ for $p$ a prime. Let $\ell_{k,m,n}$ denote the length of the module $\bK[x,y]/( x^m,y^n,(x+y)^k)$, where $\bK$ is a field of either characteristic zero or characteristic $p$. Then
\begin{enumerate}
\item If $n\geq k+m$, then
    \[\ell_{k,m,n} = km.\]
\item If $k\geq m+n$, then 
    \[\ell_{k,m,n}=mn.\]
\item If $m\geq n+k$, then
    \[\ell_{k,m,n} = kn.\]
\item If we are not in any of the above cases (i.e., if $k+m>n$ and $m+n> k$ and $m > n+k$), then
    \[
    \ell_{k,m,n} = kn+km+nm- \left\lfloor \frac{(k+n+m)^2}{4} \right\rfloor.
    \]
\end{enumerate}
\end{theorem}

\begin{remark}
As in the proof of \Cref{monomial-len-lemma}, there is overlap between cases (a), (b), and (c) above, but this only occurs when one or more of $k,n,m$ are zero.
\end{remark}

\begin{proof}
Note that since $\bK[x,y]/(x^m,y^n,(x+y)^k) \cong \bK[x,y,z]/( x^m,y^n,z^k,x+y+z)$, the length is unchanged by permutating $m,n,k$. Thus without loss of generality we assume than $m+k\leq p$ and $m\leq k$.

Consider the Gr\"obner basis found in \Cref{prop-Grobnerbasis1} for when $m\leq k$. Recall that the element $f_{2t+1}$ has leading term $\frac{\binom{k+t}{i-t}}{\binom{m-1}{t}}x^{m-1-t}y^{k+2t+1-m}$.
Thus the initial term ideal of $I_{k,m,n}$ is precisely
\[
\LT (I_{k,m,n}) = \left(x^{m-1-t}y^{k-m+1+2t}  \ | \ 0\leq t \leq m-1\right) + ( x^m,y^n) .
\]
These are exactly the monomials that appeared in \Cref{monomial-len-lemma}, so applying this lemma with $\alpha = k$, $\beta = m$, and $\eta = n$ gives the following cases:
\begin{enumerate}[label=(\arabic*)]
\item If $m + n \leq k$, then
\[
\ell(\bK[x,y]/J) = m n.
\]

\item If $k+m \leq n$, then
\[
\ell(\bK[x,y]/J) = m k.
\]

\item If $m + n > k$ and $k+m >n$, then
\[
\ell(\bK[x,y]/J) = m n - \left\lfloor \frac{(m+n-k)^2}{4}\right\rfloor.
\]
\end{enumerate}
The length formula in case (3) can be re-written as
\begin{align*}
m n - \left\lfloor \frac{(m+n-k)^2}{4}\right\rfloor &= \left\lceil \frac{4mn}{4} 
- \frac{m^2+n^2+k^2-2mk-2nk+2mn}{4}     \right\rceil \\
    &= \left\lceil \frac{2mk+2nk+2mn - m^2-n^2-k^2}{4}\right\rceil \\
    &= \left\lceil (mk+nk+mn)-\frac{(m+n+k)^2}{4}\right\rceil\\
    &= mk+nk+mn - \left\lfloor \frac{(m+n+k)^2}{4}\right\rfloor.
\end{align*}
Via permuting $m,n,k$, these give the cases in the original theorem statement.
\end{proof}

\section{\texorpdfstring{Lengths and Gr\"obner bases when $f=x^ay^b(x+y)^c$}{Lengths and Groebner bases when f=(x{\textasciicircum}a)(y{\textasciicircum}b)(x+y){\textasciicircum}c}}
\label{sec:generalized-abc-case}

We now return to our goal of computing the lengths of the modules of the form
\[ \mathbb{F}_p [x,y]_{(x,y)}/(x^M, y^M, f^K),\]
where $f = x^a y^b (x+y)^c$ for non-negative integers $a$, $b$ and $c$ and positive integers $M, K$ with $M \leq p$. In order to make use of our previous computations in the simplified case, we will need the following lemma.
\begin{lemma} \label{initialidealreduction}
Suppose that $J = ( x^m, x^{a}y^bg, y^n)$ for some polynomial $g$ in $\bK[x,y]$. Define $a'=\min(a,m)$ and $b'=\min(b,n)$, and let $J'=( x^{m-a'}, x^{a-a'}y^{b-b'}g, y^{n-b'})$. Then
\[
\LT(J) = ( x^{m}, y^n) + x^{a'}y^{b'}\LT(J').
\]
In particular, we can easily find a Gr\"obner basis or colength for $J$ if we have the corresponding info for $J'$:
\begin{enumerate}
\item If $g_1, \ldots, g_s$ is a Gr\"obner basis for $J'$, then 
\[
\{x^m, y^n\}\cup \{x^{a'}y^{b'}g_i \: 1\leq i \leq s \}
\]
is a Gr\"obner basis for $J$.
\item The colength of $J$ is
\[
\ell(\bK[x,y]/J) = nb' + (m-b')a'+ \ell(\bK[x,y]/J').
\]
\end{enumerate}
\end{lemma}
\begin{proof}
First, if either $m\leq a$ or $n\leq b$, then in fact $J=(x^m,y^n)$ and $J'=( 1 )$, so the statement clearly holds. Now assume $m>a$ and $n>b$, so that $J' = (x^{m-a},\, g,\, y^{n-b})$.

First, we show that the RHS is contained in the LHS. The monomials $x^m$ and $y^n$ are both clearly contained in $\LT(J)$. Further, consider any $h = h_0x^{m-a}+h_1g + h_2y^{n-b}\in J'$, so that $\LT(h)\in \LT(J')$. Thus $\LT(x^ay^bh) = (h_0y^b)x^m + (x^ay^bh_1)g + (h_2x^a)y^n\in \LT(J)$.

Now, we show the LHS is contained in the RHS.
Take any $h\in J$, writing $h=h_0x^m+ h_1x^ay^bg+h_2y^n$. Use the Euclidean algorithm to write $h_0 = q_0 y^b + r_0$ and $h_2 = q_2x^a+r_2$, so that
\[
h=x^ay^b(q_0x^m+h_1g+q_2y^n) + r_0x^m+ r_2y^n.
\]
In particular, no term of $r_0x^m$ can cancel with any other term, because no term of $r_0$ is divisible by $y^b$; and similarly for $r_2y^n$. Thus we have three cases:
\begin{itemize}
\item $\LT(h) = \LT(r_0x^m)$: Then $\LT(h)\in ( x^m,y^n)$.
\item $\LT(h) = \LT(r_2y^n)$: Then $\LT(h)\in ( x^m,y^n)$.
\item $\LT(h) = \LT(x^ay^b(q_0x^m+h_1g+q_2y^n))$: Since $q_0x^m+h_1g+q_2y^n\in J'$, we have $\LT(h) \in x^ay^b\LT(J')$.
\end{itemize}

For the Gr\"obner basis conclusion, first note that the proposed basis elements are all contained in $J$. By the above discussion, their initial terms also generate the initial ideal of $J$ and thus form a Gr\"obner basis as desired.

\begin{figure}[ht]
    \begin{center}
    \includegraphics{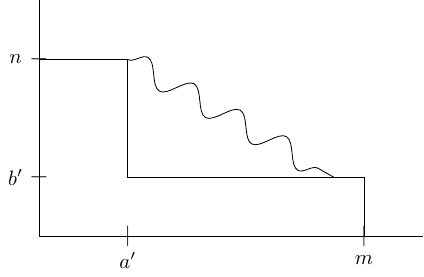}
    \end{center}
    \caption{The diagram of monomials in $\bK[x,y]/\LT(J)$. The diagram for $\bK[x,y]/\LT(J')$ is the ``triangular'' region in the upper right.}
    \label{fig:factorideal}
\end{figure}
For the length conclusion, we use that $\bK[x,y]/J$  and $\bK[x,y]/\LT(J)$ have the same Hilbert function, and that the colength of a monomial ideal can be determined by ``counting boxes''. More specifically, multiplying by $x^{a'}y^{b'}$ shifts the the diagram for $\bK[x,y]/\LT(J')$ up by $b'$ and to the right by $a'$. See \Cref{fig:factorideal} for an illustration.
\end{proof}

We will now use \Cref{initialidealreduction} to extend the Gr\"obner basis and lengths we computed in \Cref{sec:simplified-f-case} when $f=x+y$ to the case when $f=x^ay^b(x+y)^c$.

\begin{proposition}
\label{prop:Grobnerbasis}
Consider the ideal $I=(x^M, y^N, f^K)$ in $\bK[x,y]$, where $f=x^ay^b(x+y)^c$ and where $M,N,K,a,b,c$ are all non-negative integers.
Assume that the field $\bK$ has either characteristic zero, or characteristic $p$ where $M+(c-a)K \leq p$.
\begin{enumerate}
\item If either $aK \geq M$ or $bK\geq N$, then
\[
\{x^M,y^N\}
\]
is a Gr\"obner basis for $I$.

\item If $aK<M, bK<N$, and $(a+c)K\geq M$, then
\[
\{x^M,\, y^N\} \cup \left\{ 
H_t \: 0\leq t < M-aK   \right\}
\]
is a Gr\"obner basis for $I$, where
\[
H_t = 
\sum_{j=0}^{M-aK-t-1} \left(\prod_{\ell=1}^t\frac{M-aK-j-\ell}{M-aK-\ell}\right) \binom{cK+t}{j} x^{aK+j}y^{(b+c)K-j+t} .
\]
Thus $I$ has initial ideal
\[
\LT(I) = (x^M,\,  y^N) +
    \left( x^{M-t-1}y^{(a+b+c)K-M+2t+1} \: 0\leq t < M-aK \right)
\]

\item Finally, if $aK<M$, $bK<N$, and $(a+c)K<M$ then
\[
\{x^M,\, y^N,\, x^{aK}y^{bK}(x+y)^{cK}\} \cup 
\left\{ L_t \: 0\leq t < cK 
\right\}
\]
is a Gr\"obner basis for $I$, where
\[
L_t = \sum_{i=0}^{cK-1-t}\left(\sum_{j=i}^{cK-1-t} (-1)^{M-aK-j-t} \left(\frac{\prod_{\ell=1}^t(cK-j-\ell)}{\prod_{\ell=1}^t(cK-\ell)}\right) \binom{M-aK+t}{j}\binom{j}{i} \right)
x^{i+aK} y^{M+(b-a)K-i+t}.
\]
Thus $I$ has initial ideal
\[
\LT(I) = ( x^M,\, y^N,\, x^{(a+c)K}y^{bK}) 
+ \left( x^{(a+c)K-t-1} y^{M+(b-a-c)K+2t+1}      \: 0\leq t < cK \right).
\]
\end{enumerate}
\end{proposition}
\begin{proof}
The first case is clear. For the second case, combine \Cref{prop-Grobnerbasis1} and \Cref{initialidealreduction}.
For the last case, combine \Cref{prop-Grobnerbasis2} and \Cref{initialidealreduction}. The constraint inequalities come from setting $m=M-aK$, $n=M-bK$, and $k=cK$.
\end{proof}

We can now compute our desired lengths. 

\begin{theorem}
\label{thm-lengths-unified}
Suppose $f=x^ay^b(x+y)^c$ is a polynomial in $\bK[x,y]$ for non-negative integers $a,b,c$, and let $M,K$ be fixed positive integers such that $aK,bK \leq M$. Assume the field $\bK$ has either characteristic $0$, or has characteristic $p$ where
\[\min\{M+(c-a)K,\ M+(c-b)K,\ 2M-(a+b)K\}\leq p.\] 
Let $\ell_{K,M}$ denote the length of the module $\bK[x,y]/( x^M,y^M,f^K)$.
\begin{enumerate}
\item If $a\geq b+c$, then
    \[\ell_{K,M} = (a+b+c)KM-a(b+c)K^2.\]
\item If $(a+b+c)K\geq 2M$, then
    \[\ell_{K,M} = M^2.\]
\item If $b\geq a+c$ then
    \[\ell_{K,M} = (a+b+c)KM-b(a+c)K^2.\]
\item If we are not in cases (a)--(c), then
    \[\ell_{K,M} = (a+b+c)KM - \left\lfloor \frac{(a+b+c)^2K^2}{4} \right\rfloor.\]
\end{enumerate}
\end{theorem}
\begin{proof}
\Cref{initialidealreduction} using $m=M-aK$, $n=M-bK$, and $k=cK$ gives 
\[
\ell_{K,M} = M(bK)+(M-bK)aK+\ell_{cK,M-aK,M-bK} = (a+b)KM -abK^2+\ell_{cK,M-aK,M-bK}.\] 
Now we use \Cref{unified-simpler-lengths} to get $\ell_{cK,M-aK,M-bK}$ and then simplify:
\begin{enumerate}
\item The condition $(M-bK) \geq cK+(M-aK)$ is equivalent to $a\geq b+c$, and gives $\ell_{cK,M-aK,M-bK} = cK(M-aK)$. Therefore
\[
    \ell_{K,M} = (a+b+c)KM - a(b+c)K^2.
\]

\item The condition $cK\geq (M-aK)+(M-bK)$ is equivalent to $(a+b+c)K\geq 2M$, and gives $\ell_{cK,M-aK,M-bK} = (M-aK)(M-bK)$. Therefore
\[
    \ell_{K,M} = M^2.
\]

\item The condition $(M-aK)\geq (M-bK)+cK$ is equivalent to $b\geq a+c$, and gives $\ell_{cK,M-aK,M-bK} = cK(M-bK)$. Therefore
\[
    \ell_{K,M} = (a+b+c)KM - b(a+c)K^2.
\]

\item If we are not in any of these cases, then $\ell_{cK,M-aK,M-bK}$ is
\begin{align*}
     &cK(M-bK)+cK(M-aK)+(M-aK)(M-bK)- \left\lfloor \frac{(cK+(M-bK)+(M-aK))^2}{4}      \right\rfloor\\
    &= M^2 + (2c-a-b)KM +(ab-ac-bc)K^2 - \left\lfloor \frac{4M^2+4(c-a-b)KM + (c-a-b)^2K^2}{4}      \right\rfloor.
\end{align*}
Therefore
\begin{align*}
\ell_{K,M} &= M^2 +2cKM -c(a+b)K^2 - \left\lfloor \frac{4M^2+4(c-a-b)KM + (c-a-b)^2K^2}{4} \right\rfloor\\
    &= (a+b+c)KM - \left\lfloor \frac{(a+b+c)^2K^2}{4} \right\rfloor. \qedhere
\end{align*}
\end{enumerate}
\end{proof}

\begin{remark} \label{rem:charppermutingabc}
    When $M = p$ and $\bK$ has characteristic $p>0$ in \Cref{thm-lengths-unified}, the hypotheses can be simplified: in this case, the ideals $(x^p, y ^p, (x^a y ^b (x+y) ^c)^K)$ and $(x^p, y ^p, (x^c y ^b (x+y) ^a)^K)$ are related by a linear change of coordinates, and hence have the same co-length. Thus, we may permute $a,b$, and $c$ to assume that $c \leq a $, in which case, the hypothesis now becomes simply $ aK \leq p$ and $bK \leq p$, with no need to consider the minimum. 
\end{remark}

\section{Connections to Lefschetz properties}
\label{sec:wlp}

We now present an alternate perspective on computing the lengths in \Cref{sec:simplified-lengths}. 
In the setting of \Cref{sec:simplified-f-case}, we can re-express our length $\ell_{k,m,n}$ as 
\[
\ell_{k,m,n} = \ell\left(\bK[x,y]/( x^m,y^n,(x+y)^k) \right) 
= \ell\left(\bK[x,y,z]/( x^m,y^n,z^k,x+y+z) \right).
\]
This suggests a connection to the \emph{weak Lefschetz property} (WLP): recall that a standard graded Artinian $F$-algebra $S$ has the WLP if there is a linear form $L\in S_1$ such that for all degrees $j$, the multiplication map $\times L: S_{j-1} \to S_{j}$ has maximal rank. 
Such an \(L\) is a \emph{weak Lefschetz element}. Further, in the case that $S=\bK[x_1,\ldots, x_n]/I$ for $I$ monomial, it suffices to check whether $L=x_1+ \cdots + x_n$ has this property \cite[Prop.~4.3]{Lundqvist+Nicklasson.19}. 

\begin{lemma}
Let $S$ be a standard graded Artinian, Gorenstein $\bK$-algebra, where $\bK$ is an algebraically closed field. 
If $L$ is a weak Lefschetz element for $S$, then
\[
\ell(S/LS) = \dim_\bK(S_{\lfloor t/2\rfloor}),
\]
where $t$ is the socle degree of $S$. 
\end{lemma}
\begin{proof}
Since $L$ is homogeneous, we can decompose
\[
\ell(S/LS) = \dim_\bK \coker(\begin{tikzcd}S\rar["\times L"] &S\end{tikzcd}) 
= \sum_{j=0}^t \dim_\bK \coker(\begin{tikzcd}S_{j-1}\rar["\times L"] &S_j\end{tikzcd}). 
\]
Since $L$ is a weak Lefschetz element, these maps are injective up until some degree, from which point on they switch to becoming surjective, by \cite[Prop.~2.1]{Migliore+etal.11}. 
Since $S$ is standard graded Gorenstein, $\dim_\bK S_j = \dim_\bK S_{t-j}$ and so this symmetry forces the switch point to be at $\lfloor t/2\rfloor$. So
\begin{align*}
\ell(S/LS) &= \sum_{j=0}^{\lfloor t/2\rfloor} \dim_\bK \coker(\begin{tikzcd}[ampersand replacement=\&]S_{j-1}\rar["\times L"] \&S_j\end{tikzcd}) = \sum_{j=0}^{\lfloor t/2\rfloor} \left(\dim_\bK S_j - \dim_\bK S_{j-1}\right) \\
&= \dim_\bK S_{\lfloor t/2\rfloor} - \dim_\bK S_{-1} = \dim_\bK S_{\lfloor t/2\rfloor}.\qedhere
\end{align*}
\end{proof}

Consider the following result of Lundqvist and Nicklasson (which extends an earlier result of Cook~II \cite[Prop.~3.5]{Cook.12}), stated here in the three variable case:
\begin{theorem}[{\cite[Thm.~4.4]{Lundqvist+Nicklasson.19}}]
\label{char-p-WLP}
Let $R=\bK[x_0,x_1,x_2]$, where $\bK$ is a field of characteristic $p>0$. Let $d_0\geq d_1\geq d_2\geq 1$ and let $t=\sum_{i=0}^2 (d_i-1)$. If $\max(p,d_0)>\frac{t+1}{2}$, then $R/( x_0^{d_0}, x_1^{d_1}, x_2^{d_2})$ has the WLP.
\end{theorem}

We can now use this to compute $\ell_{k,m,n}$ directly, as the following result shows.

\begin{theorem} \label{thm:lenghtsvialefschetz}
Suppose $\{k,m,n\} = \{d_0\geq d_1\geq d_2\}$, and let $t=k+m+n-3$. Suppose $\bK$ is a field of characteristic $p$ where $p$ is a positive prime, and assume $0\leq m,n\leq p-k$. Then
\begin{enumerate}
\item If $d_0 \leq \lceil \frac{t}{2}\rceil$, i.e., if $2d_0\leq t+1$, then
\[
\ell_{k,m,n} = \binom{\lfloor t/2\rfloor+2}{2} - \sum_{i=0}^2 \binom{\lfloor t/2\rfloor - d_i +2}{2}.
\]
\item If $d_0 > \lceil \frac{t}{2}\rceil$, i.e., if $2d_0>t+1$, then
\[
\ell_{k,m,n} = d_1d_2.
\]
\end{enumerate}
\end{theorem}

\begin{proof}The starting assumption of $0\leq m,n\leq p-k$ means that $t=m+n+k-3 \leq p+n-3\leq 2p-k-3$. But then \[
\left\lceil \frac{t+1}{2}\right\rceil \leq \left\lceil \frac{2p-k-2}{2}\right\rceil = p-1-\left\lfloor \frac{k}{2}\right\rfloor < p.
\]
Thus $S=\bK[x,y,z]/ (x^m,y^n,z^k)$ has the WLP. From here, we handle each case separately.
\begin{enumerate}
\item Since $S$ has the WLP, this means
\[
\ell_{k,m,n} = \ell(S/(x+y+z)S) = \dim_\bK (S_{\lfloor t/2\rfloor}).
\]
Further, we in general know that $\dim_\bK(\bK[x,y,z])_i = \binom{i+2}{2}$, and by inclusion-exclusion, that $\dim_\bK(( x_0^{d_0},x_1^{d_1},x_2^{d_2}))_i$ is 
\[
\left(\sum_{j=0}^2\binom{i-d_j+2}{2}
    \right)
    - \left(\sum_{0\leq j < j'\leq 2}\binom{i-d_j-d_{j'}+2}{2}
    \right)
    + \binom{i-d_0-d_1-d_2+2}{2}.
\]
However, since $d_2\leq d_1\leq d_0 \leq \lceil t/2\rceil$,  this means that 
\[
\left\lfloor \frac{t}{2}\right\rfloor - d_j - d_{j'}+2 
= 
\left\lfloor \frac{t}{2}\right\rfloor +d_{j''}-t-1 = d_{j''} - \left\lceil \frac{t}{2}\right\rceil - 1<0.
\]
Thus
\[
\dim_\bK(( x_0^{d_0},x_1^{d_1},x_2^{d_2}))_{\lfloor t/2\rfloor} 
= \left(\sum_{j=0}^2\binom{\lfloor t/2\rfloor-d_j+2}{2}\right).
\]
Now subtracting gives our desired result.

\item In this case we return to our original two-variable perspective. We can by change of variables assume that $k=d_0$, and rewrite the inequality as $k > \lceil \frac{k+m+n-3}{2}\rceil$. This implies that $2k>(k+m+n-3)+1$, i.e., that $k\geq m+n-1$. This then ensures that $(x+y)^k \in ( x^m, y^n)$, and so 
\[\ell_{k,m,n} = \ell(\bK[x,y]/( x^m,y^n)) = mn = d_1d_2. \qedhere\]
\end{enumerate}
\end{proof}

\begin{remark}
By expanding out the formula in case~(a), we get that the length $\ell_{k,m,n}$ is
\[
-\left\lfloor \frac{d_0+d_1+d_2+1}{2}\right\rfloor^2 + (d_0+d_1+d_2+1) \left\lfloor \frac{d_0+d_1+d_2+1}{2}\right\rfloor - \frac{(d_0+1)d_0+(d_1+1)d_1+(d_2+1)d_2}{2}.
\]
Since in general $w\cdot\lfloor n/2\rfloor - \lfloor w/2\rfloor^2 = \lfloor w^2/4\rfloor$, we can combine and simplify the above to $\left\lfloor \frac{4d_1d_2-(d_0-d_1-d_2)^2+1}{4}\right\rfloor$, which can be directly checked to be the same as the formula we got in \Cref{unified-simpler-lengths}.
\end{remark}

\begin{remark}
The referee pointed out to us that Han computes the same length formulas as above, but with the assumption that $(x+y)^k$ is a non-zero-divisor on the degree $\frac{\lfloor m+n-k-2\rfloor}{2}$ component of $\mathbb K[x,y]/(x^m,y^n)$, see \cite[Thm.~2.3]{HanThesis}. Recently, Nicklasson \cite[Thm.~3.2]{Nicklasson.18} has classified monomial complete intersections in two-variables with the \emph{strong Lefschetz property (SLP)} (also see Cook~II \cite[Cor.~4.8]{Cook.12} for the characteristic $2$ case), which confirms Han's assumption whenever $m,n,k,p$ satisfy the hypothesis of \Cref{thm:lenghtsvialefschetz}. This yields a third computation of the lengths in \Cref{unified-simpler-lengths}.
\end{remark}

\section{\texorpdfstring{Limit $F$-signature}{Limit F-signature}}
\label{sec:limit-f-signature}
In this section, we will use the computation of the lengths from \Cref{thm-lengths-unified} to compute the limit $F$-signature functions of the polynomials $f= x^a y^b (x+y)^c$ and $g = x^a y^b (x^u + y^v)^c $ for non-negative integers $a$, $b$, $c$, $u$, and $v$. In what follows, by abuse of notation we view the polynomials in rings of varying characteristics, which is well-defined since $f$ and $g$ have coefficients in $\ZZ$.

\begin{theorem} \label{thm:limFsigfunction1}
    Assume $a \geq b \geq c \geq 0$ are non-negative integers. Let $\psi(t)$ be the limit $F$-signature function of $f = x^a y^b (x+y)^c$, so that 
    \[
    \psi(t) = \lim_{p\to\infty} s_p(\bF_p[x,y]_{( x,y)}, f^t).
    \]
    Let $\lambda = \lct(f)$. Then, for any real number $t \geq 0$, we have:
      \begin{enumerate}
\item If $t \geq \lambda$, then $\psi(t) = 0$. 
\item If $t< \lambda$ and $a\geq b+c$, then $\psi(t) = (1 - at) (1 - (b+c)t)$.
\item If $t < \lambda$ and $a < b+c$, then $\psi(t) = \left( 1 - \frac{(a+b+c)}{2}t \right)^2$.
\end{enumerate}
In other words, when $a \geq b+c$, we have
\[
\psi(t) = \begin{cases}
(1-at)(1-(b+c)t) & 0 \leq t < \frac{1}{a} \\
0 & t \geq \frac{1}{a}
\end{cases},
\]
and when $a < b+c$ we have
\[
\psi(t) = \begin{cases}
\left(1-\frac{(a+b+c)}{2}t\right)^2 & 0 \leq t < \frac{2}{a+b+c} \\
0 & t \geq \frac{2}{a+b+c}
\end{cases}.
\]
\end{theorem}

We first observe what happens in the $c=0$ case before diving into the full proof.

\begin{lemma} \label{lem:monomialcasecomputation}
    Fix non-negative integers $a, b \geq 0$. Then, for any non-negative real parameter $t \geq 0$, the $F$-signature function of the pair $(\mathbb{F}_p [x,y]_{(x,y)} , (x^{a} y^{b})^t)$, denoted by $\psi_p (t)$, is given by
    \[ \psi_p(t) = \begin{cases} 
    (1-at)(1-bt) & 0\leq t < \min\{\frac{1}{a},\ \frac{1}{b}\} \\
    0 & t \geq \min\{\frac{1}{a},\ \frac{1}{b}\}
    \end{cases}.\]
\end{lemma}

\begin{proof}
See \cite[Example 4.19]{BlickleSchwedeTuckerFSigPairs1}.
\end{proof}

\begin{proof}[Proof of \Cref{thm:limFsigfunction1}]
To check that our proposed piecewise formula is indeed the same as our proposed list of cases, observe first by an easy computation after blowing up the origin that $\lambda = \lct(f) =  \min \{\frac{1}{a}, \frac{2}{a+b+c} \}.$ 
Therefore, we see that $1/a$ is the minimum exactly when $a+b+c\leq 2a$, equivalently when $a\geq b+c$. 

   Let $\phi(t)$ denote the candidate formula for the limit $F$-signature. In other words, if $a \leq b+c$ then $\lambda = \frac{2}{a+b+c}$ and for $ t \geq 0$,
    \[ \phi(t) = \begin{cases}
    \left(1-\frac{(a+b+c)}{2}t\right)^2 & 0 \leq t < \frac{2}{a+b+c} \\
    0 & t \geq \frac{2}{a+b+c}
    \end{cases}.\]
    Similarly, if $a> b+c$ then $\lambda = \frac{1}{a}$ and for $t \geq 0$, we have
    \[\phi(t) = \begin{cases}
    (1-at)(1-(b+c)t) & 0 \leq t < \frac{1}{a} \\
    0 & t \geq \frac{1}{a}
    \end{cases}. \]
    Note that in either case, the function $\phi(t)$ is continuous and non-increasing. Our goal is to prove that $\phi(t)=\psi(t)$. For notational convenience, let  
\[ \psi_{p} (t)  =   s_{p}(\mathbb{F}_p[x,y]_{(x,y)}, \, f ^{t}).  \]

If $c=0$, then $a\geq b$ and we are in the case of \Cref{{lem:monomialcasecomputation}} which clearly matches our proposed $\phi(t)$. So now assume $c > 0$.

    For any $t >0$ and prime numbers $p$, write $r_p = \lfloor tp \rfloor$ so that $\lim_{p \to \infty} \frac{r_p} {p} = t$.
    By \Cref{lem:limitFsigofhypersurfaces}, to compute the limit $F$-signature function $\psi(t)$ (and show that it exists), it suffices to show that $\lim_{p \to \infty} \psi_p (\frac{r_p}{p}) = \phi(t)$. This is clear when $t \geq \lambda$, since then $t \geq \fpt(f)$ (see \Cref{def:Fpurethreshold} and \Cref{fpt-lct-properties}) and thus $\psi_p (t) = 0$ for all $p$.
    So we now fix a $t$ with $\lambda > t >0$, and define $r_p$ as above. 
    
    Let $S_p$ denote the ring $\mathbb{F}_p [x,y]_{(x,y)}$. Then by \Cref{equation:Fsiglength}, we have
    \[ \psi_p (\frac{r_p}{p}) = \frac{1}{p^2} \ell_{S_p}(S_p/(\mathfrak{m}^{[p]} : f^{r_p})).  \]
    By the exact sequence
    \[0 \to S_p/(\mathfrak{m}^{[p]} : f^{r_p}) \to S_p /\mathfrak{m}^{[p]} \to S_p/(\mathfrak{m}^{[p]}, f^{r_p}) \to 0 \]
    where the left map is given by multiplication by $f^{r_p}$, we see that
    \[\psi_p(\frac{r_p}{p}) = 1 - \frac{1}{p^2} \ell_{S_p}(S_p/(x^p, y^p, f^{r_p})). \]
    Thus, to prove the theorem, it is suffices to show that 
    \begin{equation} \label{eqn:lenghtlimit} 
    \lim_{p \to \infty} \frac{\ell_{S_p}(S_p/(x^p,y^p, f^{r_p}))}{p^2} = 1 - \phi(t) .\end{equation}
    
    We will check this using the lengths computed in \Cref{thm-lengths-unified}.
    Since we have assumed that $0<t < \lambda\leq \frac{1}{a}$, for all $p \gg 0$ we have $r_p >0$ and also $r_p \, b \leq r_p \, a \leq p$. 
    Thus we may apply \Cref{thm-lengths-unified} with $M = p$ and $K = r_p$ (see \Cref{rem:charppermutingabc}).
    
    First consider the case when $a\geq b+c$, which is part~(a) of \Cref{thm-lengths-unified}. Then the length formula gives
    \[\lim_{p \to \infty} \frac{\ell_{S_p}(S_p/(x^p,y^p, f^{r_p}))}{p^2} = \lim_{p \to \infty} \frac{(a+b+c)r_p \, p - a(b+c)r_p ^2}{p^2} = (a+b+c)t - a(b+c)t^2,  \]
    which is easily seen to be equal to $1 - (1 - at) (1 - (b+c)t)$.

    Next, assume $a<b+c$ and $t<\lambda$. Since $t<\frac{2}{a+b+c}$, this means for all $p\gg 0$, we have $r_p(a+b+c) < 2p$. Since also $a\geq b$, we must be in part~(d) of \Cref{thm-lengths-unified}. Thus we have
    \begin{align*}
    \lim_{p\to\infty} \frac{\ell_{S_p}(S_p/( x^p,y^p,f^{r_p}))}{p^2} &= \lim_{p\to\infty} \frac{(a+b+c)r_p p -\left\lfloor \frac{(a+b+c)^2r_p^2}{4} \right\rfloor}{p^2} \\
    &= (a+b+c)t - \frac{(a+b+c)^2t^2}{4} \\
    &= 1 - \left(1-\frac{(a+b+c)}{2}t\right)^2
    \end{align*}
as required. This completes the proof.
\end{proof}

\begin{remark} \label{rem:rationalexponent}
    Given three positive rational numbers $r_1$, $r_2$, and $r_3$, it is possible to choose integers $a,b,c$ and a rational number $t>0$ such that $r_1 = at$, $r_2 = bt$, and $r_3 = ct$. Thus, \Cref{thm:limFsigfunction1} can be used to compute the limit $F$-signature of $x^{r_1} y^{r_2} (x+y)^{r_3}$ for any three positive rational numbers $r_1$, $r_2$, $r_3$. Moreover, the cases in the statement of \Cref{thm:limFsigfunction1} depend only on the $r_i$ and not on the specific choices of $a$, $b$, $c$, and $t$. Therefore, if we assume $r_1 \geq r_2 \geq r_3$ and let $s_\infty$ denote the limit $F$-signature of the pair $(\ZZ[x,y], x^{r_1} y^{r_2} (x+y) ^{r_3})$, then \Cref{thm:limFsigfunction1} gives:
      \begin{enumerate}
\item If $r_1 + r_2 + r_3 \geq 2$ or $r_1 \geq 1$, then $s_\infty = 0$. 
\item If $r_1 <1$, $r_1 + r_2 + r_3 <2$, and $r_1\geq r_2 + r_3$, then $s_\infty = (1 - r_1) (1 - (r_2 + r_3))$.
\item If $r_1 <1$, $r_1 + r_2 + r_3 <2$, and $r_1 \leq r_2 + r_3$, then $s_\infty = \left( 1 - \frac{(r_1 + r_2 + r_3)}{2} \right)^2$.
\end{enumerate}
\end{remark}

\begin{remark}
\label{fsigfunction-pdenom}
While proving \Cref{thm:limFsigfunction1}, we in fact computed the $F$-signature function $\psi_p$ itself for certain values of $t$. For polynomial $f$ and parameters $a\geq b\geq c$ as in the theorem statement, and for any integer $r$ and prime $p$ with $0<\frac{r}{p}< \min\{\frac{1}{a},\frac{2}{a+b+c}\}$, we have
\[
\psi_p(\frac{r}{p}) = 
\begin{cases}
1 - (a+b+c)\frac{r}{p}+a(b+c)\frac{r^2}{p^2} & a\geq b+c \\
1 - (a+b+c)\frac{r}{p} +\frac{1}{p^2}\lfloor \frac{(a+b+c)^2r^2}{4}\rfloor & a<b+c
\end{cases}.
\]
\end{remark}

\begin{remark}
    The $F$-signature functions $\psi_p (t)$ from \Cref{thm:limFsigfunction1} do not always stabilize to the limit $\psi (t)$ for large values of $p$ on any sub-interval of $(0,1)$. Indeed, take $a,b,c$ to be positive integers such that $a <  b+c$ and $a + b + c$ is odd. Then, in any sub-interval $\emptyset \neq I \subset (0,1)$,  we can find primes $p \gg 0 $ and odd integers $r_p$ such that $\frac{r_p}{p} \in I$, so that by comparing \Cref{fsigfunction-pdenom} and \Cref{thm:limFsigfunction1} we see that $\psi_p (\frac{r_p}{p}) \neq \psi (\frac{r_p}{p})$. We do not know if $\psi_p (t)$ stabilizes for $p \gg 0$ for a fixed $t$, but we expect the answer to be negative. See \cite[Theorem 7.1]{caminatashidelertuckerzermandiagonalhypersurfaces} for a similar phenomenon.   
\end{remark}

\begin{theorem}
\label{thm:limitfsig}
    Let $g = x^a y^b (x^u + y^v) ^c $ for  non-negative integers $a$, $b$, $c$, $u$, and $v$. Assume that $u$, $v$, and $c$ are positive. We consider $g$ as a polynomial in $\ZZ[x,y]$, so that we may consider the reduction of $g$ to characteristic $p >0$ for each prime $p$.
    Set $\lambda_0 = \frac{\frac{1}{u} + \frac{1}{v}}{\frac{a}{u} + \frac{b}{v} + c}$, and $\lambda = \min \{\frac{1}{a}, \frac{1}{b}, \frac{1}{c}, \lambda_0 \}$.
    Let $\psi(t)$ be the limit $F$-signature function of $g$, which is defined by
    \[ \psi(t):= \lim_{p \to \infty} s_p(\mathbb{F}_p[x,y], g^t).  \]
    Then, for $t < \lambda$, $\psi(t)$ is given by the following formula:
    \begin{enumerate}
    \label{equation:formulaforlimFsig1}
        \item If $tav +u \geq tbu+v+cuvt$, then $\psi(t) = ( 1 -at) (1- (b +cv)t)$.
        \item If $tbu +v \geq tav+u+cuvt$, then $\psi(t) = (1-bt) (1 - (a +cu)t)$.
        \item If $tcuv + u + v \geq  tav  + tbu +2uv $, then $\psi(t) = (1-ct) (v(1-at) + u (1-bt) -uv)$.
        \item Otherwise, $\psi(t) = \frac{1}{4uv} (u+v - t(av + bu+ cuv))^2$.
    \end{enumerate}
\end{theorem}

\begin{remark}
If instead any one of $u$, $v$, or $c$ is zero, then the $F$-signature function of the pair $(\mathbb{F}_p [x,y]_{(x,y)},\, g = x^a y^b (x^u + y^v)^c)$ reduces to the monomial case from \Cref{lem:monomialcasecomputation}. This is because in that case $(x^u + y^v)^c$ is a unit in the localization $\mathbb{F}_p [x,y]_{(x,y)}$, so the divisor defined by $g$ is the same as the divisor defined by $x^a y^b$, and consequently, the $F$-signature functions are the same.
\end{remark}

Before proving \Cref{thm:limitfsig}, we prove some useful lemmas, the first of which will allow us to reduce to the case when $u=v=1$.

\begin{lemma} \label{uv1reduction}
    Fix non-negative integers $a$, $b$, $c$, $u$, and $v$, and fix a prime $p > uv$. Consider the polynomials
    \[ g=x^ay^b(x^u+y^v)^c,
    \quad
    \wtl g = x^{va} (x + y)^{cuv}y^{ub},
    \quad\textrm{and}\quad 
    g_0 = x^{v(u-1)} y^{u (v-1)}.\]
    Then for any $t >0$, we have
    \[ s_p(\mathbb{F}_p [x,y], g^t) =  uv \, s_p(\mathbb{F}_p [x,y], \, g_0 ^{\frac{1}{uv}} \, \wtl g^{\frac{t}{uv}}).\]
\end{lemma}

\begin{proof}
    Let $R = \mathbb{F}_p [x,y]$ and define
    \[ S = R[x', y']/ ((x')^u- x , (y')^v - y). \]
    Note that $S$ is isomorphic to the polynomial ring generated by $x'$ and $y'$ (over $\mathbb{F}_p$), and $R$ can be identified with the subring of $S$ generated by $x = (x')^{u}$ and $y = (y')^v$.
    Denoting $X = \Spec(R)$ and $Y = \Spec(S)$ and $\pi$ to be the induced map $\pi: Y \to X$, we have that $\pi$ is a finite separable map of degree $uv$ (it is separable since $p > uv$). Note that the ramification divisor $\text{Ram}_\pi$ is given by
    \[ \text{Ram}_\pi = (u-1) \divi(x') + (v-1)\divi(y'). \]
    Now, we consider the $\mathbb{Q}$-divisor 
    \[ \Delta_X = \divi (g_0 ^{\frac{1}{uv}} \wtl g^{\frac{t}{uv}}) = \frac{ta + (u-1)}{u} \divi(x) + ct \, \divi(x + y) + \frac{tb + (v-1)}{v} \divi(y).\]
    Then, using the fact that $\pi^* (\divi(x)) = u \divi (x')$, $\pi^* (\divi(y)) = v \divi (y')$, and $\pi^* (\divi(x + y)) = \divi((x')^{u} + (y')^{v})$, we compute
    \[\Delta_Y = \pi^* \Delta_X - \text{Ram}_\pi = ta \divi(x') + ct \divi((x')^u + (y')^v) + tb\divi(y') = \divi (g^t), \]
    where we view in $g$ as a polynomial in $x'$ and $y'$. Note that $\Delta_Y$ is effective. Furthermore, since $R \to S$ is a finite separable extension of normal domains and the degree is coprime to $p$, and since the corresponding map on the residue fields is an isomorphism, the number of $\text{Tr}$-summands of $S$ over $R$ is equal to $1$ by \cite[Lemma~2.13]{Carvajal-rojas-Schwede-Tucker-Fundamentalgroups}. Now, by applying \cite[Theorem 4.4]{Carvajal-rojas-Schwede-Tucker-Fundamentalgroups}, we conclude that
    \[ s_p (Y, \Delta_Y) = uv \, s_p (X, \Delta_X) \]
    as required.
\end{proof}

\begin{lemma}
    Assume $a,b$ are non-negative integers and $c,u,v$ are positive integers. The log canonical threshold of $g = x^a y^b (x^u + y^v) ^c$ is equal to
    \[ \lambda = \min \{\frac{1}{a}, \frac{1}{b}, \frac{1}{c}, \lambda_0 \}  \]
    where $\lambda_0 = \frac{\frac{1}{u} + \frac{1}{v}}{\frac{a}{u} + \frac{b}{v} + c}$.
\end{lemma}
\begin{proof}
We provide a proof using finite covers similar to the proof of \Cref{uv1reduction}. Let $g_0 = x^{u (v-1)} y ^{v (u -1)}$ and let $\tilde g = x^{va} (x+y) ^{cuv} y^{ub}$. Then the pair $(\bC[x,y], g^{t})$ is KLT if and only if $(\bC[x,y], g_0 ^{1/uv} \tilde g ^{t/uv})$ is KLT. This follows from the exact same computation as in \Cref{uv1reduction} and using the transformation rule for multiplier ideals under finite covers \cite[Theorem 8.1]{BlickleSchwedeTuckerFsingviaalterations}. Alternatively, we could also use reduction modulo $p$ to prove the statement for $F$-regularity first, in which case it follows from \Cref{uv1reduction} and the fact that the $F$-signature is positive if and only if the pair is strongly $F$-regular \cite[Theorem 3.18]{BlickleSchwedeTuckerFSigPairs1}. In any case, this shows that the log canonical threshold is given by the supremum
\[ \sup \{ t \geq 0 \, | \, \max \{\frac{u-1 + at}{u}, \frac{v-1 + bt}{v}, ct, \frac{\frac{u - 1 + at}{u} + \frac{v -1 + bt}{v} + ct} {2} \}  \leq 1 \} .\]
One checks that the four possible values for $\lambda$ as proposed correspond to the four values in the above maximum being equal to $1$. This shows that $\lambda $ is equal to the supremum as required.
\end{proof}

\begin{proof}[Proof of \Cref{thm:limitfsig}]
    First, we note that the proposed formula for $\psi(t)$ in the theorem defines a continuous, non-increasing function. Thus, by \Cref{lem:limitFsigofhypersurfaces}, it is sufficient to show that $\psi(t)$ matches with the limit $F$-signature function for rational values of $t$. Therefore, we may assume that $t$ is rational. From \Cref{uv1reduction}, for each $p > uv$ and $t \geq 0$, we have the transformation rule
    \[ s_p(\mathbb{F}_p[x,y], g^t) = uv \, s_p(\mathbb{F}_p[x,y], x^{\frac{at+u-1}{u}} y^{\frac{bt+v-1}{v}} (x+y)^{ct} ). \]
    Then, for rational number $t < \lambda$, we may compute the limit $F$-signature $\psi(t)$ using \Cref{thm:limFsigfunction1} for the right hand side (see \Cref{rem:rationalexponent}) and we obtain the following cases:
    \begin{enumerate}
        \item If $\frac{at + u-1}{u} \geq \frac{bt + v-1}{v} + ct $: In this case, we have
        \[\psi(t) = uv \, (1 - \frac{at + u-1}{u}) (1 - (\frac{bt + v-1}{v} + ct)) = (1 - at)(1 - (b+cv)t).   \]
        Note that by rearranging the inequality $\frac{at + u-1}{u} \geq \frac{bt + v-1}{v} + ct $, it is equivalent to $tav + u \geq tbu + v + cuvt$.
        
        \item If $\frac{bt + v-1}{v} \geq \frac{at + u-1}{u} + ct $: By symmetry between $(a,u)$ and $(b,v)$, in this case we have
        \[ \psi(t) = (1 - bt) (1 - (a+cu)t). \]
        \item If $ct \geq \frac{at + u-1}{u} +  \frac{bt + v-1}{v} $: In this case, we have
        \[\psi(t) = uv \, (1 - ct) (1 - (\frac{at + u-1}{u} +  \frac{bt + v-1}{v} )) = ( 1- ct) (u+v - avt - but - uv), \]
        and the last term simplifies to $v(1-at) + u (1-bt) - uv$. We also note that the condition $ct \geq \frac{at + u-1}{u} +  \frac{bt + v-1}{v} $ can be rearranged to $tcuv + u+v \geq tav + tbu + 2uv$. 
        \item Otherwise: Lastly, we note that in the complementary cases, the formula in \Cref{rem:rationalexponent} is symmetric in $r_1$, $r_2$, and $r_3$ so the relative order doesn't matter. Therefore, in the complementary cases, we have
        \[ \psi(t) = uv \, (1 - \frac{1}{2}(\frac{u+at -1}{u} + \frac{bt+v - 1}{v} + ct))^2 = \frac{1}{4uv}(u+v - (atv + btu+ ctuv))^2\]
        as desired.\qedhere
    \end{enumerate}  
\end{proof}

\section{Relation  to the normalized volume}

In this section, we provide a geometric interpretation for the formula of the limit $F$-signature obtained in \Cref{thm:limitfsig} in terms of an invariant of complex KLT singularities called the \emph{normalized volume}. Recall that KLT singularities are related to strongly $F$-regular singularities via reduction to characteristic $p$ \cite{Smith.97a, Hara.98, Mehta+Srinivas.97, Hara+Yoshida.03, Takagi.04a}. See \cite[Section 4]{SchwedeTuckerTestIdealSurvey} for an introduction and additional references. More recently, the normalized volume of KLT singularities was introduced by Li in relation to $K$-stability and the existence of K\"ahler-Einstein metrics on Fano varieties. We begin by reviewing the definition of the normalized volume:

\begin{definition} (\cite{LiMinimizingNormalizedVolumes}) \label{def:normalizedvolume}
    Let $x \in (X, \Delta)$ be a local KLT pair over $\mathbb{C}$ at a closed point $x \in X$. Then, the normalized volume $\nvol (X, \Delta, x)$ is defined as
    \[\nvol (X, \Delta, x) = \inf \{  A_{X, \Delta} (v)  ^n \, \text{vol}(v) \, | \, v \text{ divisorial valuation centered at $x$} \}. \]
    Here, $n$ denotes the dimension of $X$ and $A_{X, \Delta}$ denotes the log discrepancy of a divisorial valuation with respect to $(X, \Delta)$.
\end{definition}

Let $f = x^a \, y^b \, (x+y) ^c$ for non-negative integers $a \geq b \geq c$ and  $t \geq 0$ be a rational parameter. For pairs considered in this paper, namely of the form $(X = \Spec(\mathbb{C}[x,y]), \Delta = \divi (f^t)) $, the normalized volume of $(X, \Delta)$ was computed by Li in \cite[Section 4.1]{LiStabilityofSheaves}. In our notation, the formula from \cite{LiStabilityofSheaves} is written as:
\begin{enumerate}
    \item If $t \geq \min \{ \frac{1}{a}, \frac{2}{a+b+c} \}$, then $ \nvol (X, \Delta)  = 0 $.
    \item If $t< \min \{ \frac{1}{a}, \frac{2}{a+b+c} \}$ and $a \geq b +c$, then $\nvol (X, \Delta) = 4 \, (1 - (b+c)t) (1 - at)$.
    \item If $t < \min \{ \frac{1}{a}, \frac{2}{a+b+c} \}$ and $a < b +c$, then $\nvol (X, \Delta) = (2 - (a+b+c)t)^2.$
\end{enumerate}

By direct inspection with the formula obtained in \Cref{thm:limFsigfunction1}, we obtain the following relation between the normalized volume and the limit $F$-signature:

\begin{corollary} \label{cor:normalizedvolvsFsig}
    Let $(X, \Delta) = (\Spec(\ZZ[x,y]), \divi (f^t))$, where $f= x^a \, y^b \, (x+y) ^c$ and $t \geq 0$ is a rational parameter. Then, we have
    \begin{equation} \label{eqn:normalizedvolumevslimFsig} \frac{\nvol(X_\mathbb{C}, \Delta_\mathbb{C})}{4} = \lim_{p \to \infty} s_p (X_p, \Delta_p) \end{equation}
    where $X_\mathbb{C}, X_p$ denote the base changes of $X$ to $\mathbb{C}$ and to $\mathbb{F}_p$ respectively (and similarly for~$\Delta$).
\end{corollary}

    We thank Yuchen Liu for pointing us to this connection.

\begin{remark}
    The $F$-signature and the normalized volume satisfy many analogous properties (see \cite{Carvajal-rojas-Schwede-Tucker-Fundamentalgroups}, \cite{MaPolstraSchwedeTuckerFsigunderbirationalmaps}, \cite{Taylorinversionadjunctionfsignature}) and are expected to be related via the reduction modulo $p$ process as evidenced in \Cref{cor:normalizedvolvsFsig} (see \cite[Section~6.3.1]{LiLiuXuGuidedTourtoNormalizedVolume}).
\end{remark}
 We also note that \Cref{eqn:normalizedvolumevslimFsig} holds for all two dimensional KLT singularities $(X, x)$, when $\Delta = 0$, since they are quotient singularities. This leads to a following natural prediction for the values of the limit $F$-signature function of homogeneous polynomials in two variables in terms of the normalized volume as computed in \cite[Section 4.1]{LiStabilityofSheaves}:
 
\begin{question}
Does \Cref{eqn:normalizedvolumevslimFsig} hold for all homogeneous polynomials $f \in \bC[x,y]$ and all rational numbers $t \geq 0$?
\end{question}

\bibliographystyle{amsalpha}
\bibliography{main}

\end{document}